\documentclass[reqno]{amsart}
\usepackage{amsfonts}

\usepackage{amsmath,amssymb,amsthm,amsfonts}
\usepackage{bbm}
\usepackage{hyperref}

\usepackage{color}
\newcommand{\Red}{\textcolor{red}}

\newtheorem{lemma}{Lemma}[section]
\newtheorem{theorem}{Theorem}[section]
\newtheorem{definition}{Definition}[section]

\newtheorem{corollary}{Corollary}[section]

\numberwithin{equation}{section}

\newcommand{\dis}{\displaystyle}

\newcommand{\B}{\mathbb{B}}

\newcommand{\R}{\mathbb{R}}

\renewcommand{\S}{\mathbb{S}}
\newcommand{\T}{\mathbb{T}}

\newcommand{\highG}{\Theta}
\newcommand{\highB}{\Lambda}

\newcommand{\FP}{\mathbf{P}}

\newcommand{\FI}{\mathbf{I}}
\newcommand{\FX}{\mathbf{X}}

\newcommand{\CD}{\mathcal{D}}
\newcommand{\CE}{\mathcal{E}}
\newcommand{\CF}{\mathcal{F}}

\newcommand{\CS}{\mathcal{S}}

\newcommand{\na}{\nabla}

\newcommand{\be}{\beta}
\newcommand{\ga}{\gamma}
\newcommand{\om}{\omega}
\newcommand{\la}{\lambda}
\newcommand{\de}{\delta}

\newcommand{\pa}{\partial}
\newcommand{\ka}{\kappa}
\newcommand{\eps}{\epsilon}

\newcommand{\De}{\Delta}
\newcommand{\Ga}{\Gamma}

\begin{document}

\title[Boltzmann equation in Critical Besov Space]{Global well-posedness in spatially critical Besov space for the Boltzmann equation}

\author[R.-J. Duan]{Renjun Duan}
\address[RJD]{Department of Mathematics, The Chinese University of Hong Kong,
Shatin, Hong Kong, P.R.~China}
\email{rjduan@math.cuhk.edu.hk}

\author[S.-Q. Liu]{Shuangqian Liu}
\address[SQL]{Department of Mathematics, Jinan Unviersity, Guangzhou 510632, P.R.~China}
\email{shqliusx@163.com}

\author[J. Xu]{Jiang Xu}
\address[JX]{Department of mathematics, Nanjing
University of Aeronautics and Astronautics,
Nanjing 211106, P.R.China}
\email{jiangxu\underline{ }79@nuaa.edu.cn}

\date{\today}

\maketitle
\begin{abstract}
The unique global strong solution in the Chemin-Lerner type space to the Cauchy problem on the Boltzmann equation for hard potentials is constructed in perturbation framework. Such solution space is of critical regularity with respect to spatial variable, and it can  capture the  intrinsic property of the Botlzmann equation. For the proof of global well-posedness, we develop some new estimates  on the nonlinear  collision term through the Littlewood-Paley theory.
\end{abstract}

\tableofcontents

\thispagestyle{empty}

\section{Introduction}

There have been extensive studies of the global well-posedness for the Cauchy problem on the Boltzmann equation. Basically two kinds of global solutions can be established in terms of different approaches. One kind is the renormalized solution by the weak stability method \cite{DL}, where initial data can be of the large size and of no regularity. The uniqueness for such solution still remains open. The other kind is the perturbative solution by either the spectrum method \cite{NI,Sh,U74,U-s,UY-AA} or the energy method \cite{Guo-IUMJ,Guo-L, Liu-Yang-Yu, LY-S}, where initial data is assumed to be sufficiently close to Maxwellians. In general solutions uniquely exist in the perturbation framework.  It is a fundamental problem in the theory of the Boltzmann equation to find a function space with minimal regularity for the global existence and uniqueness of solutions. The paper aims at presenting a function space whose spatial variable belongs to the critical Besov space $B^{3/2}_{2,1}$ in dimensions three. The motivation to consider function spaces with spatially critical regularity is inspired by their many existing applications in the study of the fluid dynamical equations \cite{BCD,D1,XW}, see also  the recent work \cite{XK} on the general hyperbolic symmetric conservation laws with relaxations. Indeed it will be seen in the paper that the Boltzmann equation at the kinetic level  shares a similar dissipative structure in the so-called critical Chemin-Lerner type space (cf.~\cite{CL}) with those at the fluid level. Specifically, using the Littlewood-Paley theory, we etablish the global well-posedness of solutions in such function space for the angular cutoff hard potentials. It would remain an interesting and challenging problem to extend the current result to other situations, such as soft potentials \cite{Guo-BE-s,SG}, angular non cutoff \cite{AMUXY-ARMA,GR} and the appearance of self-consistent force \cite{DS-VPB}.

The Boltzmann equation in dimensions three (cf.~\cite{CIP-Book,Villani1}), which is used to describe the time evolution of the unknown velocity distribution function $F=F(t,x,\xi)\geq 0$ of particles with position
$x=(x_1,x_2,x_3)\in\R^3$ and velocity
$\xi=(\xi_1,\xi_2,\xi_3)\in\R^3$ at time $t\geq 0$, reads
\begin{equation}
\label{Beq}
\pa_t F +\xi\cdot \na_x F =Q(F,F).
\end{equation}
Initial data $F(0,x,\xi)=F_0(x,\xi)$ is given. $Q(\cdot,\cdot)$ is the bilinear Boltzmann collision operator, defined by
$$
Q(F,H)=\int_{\R^3} d\xi_\ast \int_{\S^2} d \omega \,|\xi-\xi_\ast|^\ga B_0(\theta)
\left(F
_\ast'H'-F_\ast H\right),
$$
where
\begin{eqnarray*}
&\dis F_\ast'=F(t,x,\xi_\ast'),\ \ H'=H(t,x,\xi'),\ \ F_\ast =F(t,x,\xi_\ast),\ \ H=H(t,x,\xi),
\end{eqnarray*}
with
 $$
 \xi'=\xi-\bigl((\xi-\xi_*)\cdot\omega\bigr)\,\omega,\ \
\xi'_*=\xi_*+\bigl((\xi-\xi_*)\cdot\omega\bigr)\,\omega,
$$
and $\theta$ is given by $\cos\theta=\omega\cdot (\xi-\xi_*)/|\xi-\xi_*|$.
The collision kernel $|\xi-\xi_\ast|^\ga B_0(\theta)$ is
determined by
the interaction law between particles. Through the paper, we assume
$0\leq \ga\leq1$ and $0\leq B_0(\theta)\leq C |\cos\theta|$ for a constant $C$, and this includes the hard potentials with angular cutoff as an example, cf.~\cite{Gr}.

In the paper, we study the solution of the Boltzmann equation \eqref{Beq} around the global Maxwellian
$$
\mu=\mu(\xi)=(2\pi)^{ -3/2 }e^{-|\xi|^2/2},
$$
which has been normalized to have zero bulk velocity and unit density and temperature.  For this purpose,  we set the
perturbation $f=f(t,x,\xi)$ by
$
F=\mu+\mu^{1/2}f$.
Then  \eqref{Beq} can be reformulated as
\begin{equation}\label{f.eq.}
\partial_{t}f+\xi\cdot\nabla_{x}f+Lf=\Gamma(f,f),
\end{equation}
with initial data $f(0,x,\xi)=f_{0}(x,\xi)$ given by $F_0=\mu+\mu^{1/2}f_0$. Here $L f$, $\Ga(f,f)$ are the linearized and nonlinear collision terms, respectively, and their precise expressions will be given later on. Recall that $L$ is nonnegative-definite on $L^2_\xi$, and $\ker L$ is spanned by five elements $\sqrt{\mu}$, $\xi_i\sqrt{\mu}$ $(1\leq i\leq 3)$, and $|\xi|^{2}\sqrt{\mu}$ in $L^2_\xi$. For later use, define the macroscopic projection of $f(t,x,\xi)$ by
\begin{equation}\label{Pf}
{\bf P}f=\left\{a(t,x)+\xi\cdot
b(t,x)+\left(|\xi|^2-3\right)c(t,x)\right\}\sqrt{\mu},
\end{equation}
where for the notational brevity we have skipped the dependence of coefficient functions on $f$. Then, the function $f(t,x,\xi)$  can be decomposed as
$
f ={\bf P}f+\{\FI- \FP\}f.
$

We now state the main result of the paper. The norms and other notations below  will be made precise in the next section. We define the energy functional and energy dissipation rate, respectively, as
\begin{eqnarray}
\mathcal {E}_T(f)\sim
\|f\|_{\widetilde{L}^{\infty}_T\widetilde{L}^{2}_{\xi}(B^{3/2}_x)},
\label{def.et}
\end{eqnarray}
and
\begin{eqnarray}
\mathcal {D}_T(f)=\|\na_x(a,b,c)\|_{\widetilde{L}^{2}_T(B^{1/2}_x)}+
\|\{\FI-\FP\}f\|_{\widetilde{L}^{2}_T\widetilde{L}^{2}_{\xi,\nu}(B^{3/2}_x)}.
\label{def.dt}
\end{eqnarray}

\begin{theorem}\label{main.th.}
There are $\eps_0>0$, $C>0$ such that if
$$
\|f_0\|_{\widetilde{L}^{2}_{\xi}(B^{3/2}_x)}\leq \eps_0,
$$
then there exists a unique global strong solution $f(t,x,\xi)$ to the Boltzmann equation (\ref{f.eq.}) with initial data $f(0,x,\xi)=f_0(x,\xi)$, satisfying
\begin{equation}\label{global.eng}
\mathcal {E}_T(f)+\mathcal {D}_T(f)\leq C\|f_0\|_{\widetilde{L}^{2}_{\xi}(B^{3/2}_x)},
\end{equation}
for any $T>0$.
Moreover, if $F_0(x,\xi)=\mu+\mu^{1/2}f_{0}(x,\xi)\geq0$, then
$F(t,x,\xi)=\mu+\mu^{1/2}f(t,x,\xi)\geq0.$
\end{theorem}

We now give a few comments on Theorem  \ref{main.th.}. Those function spaces appearing in the key inequality \eqref{global.eng} are called the Chemin-Lerner type spaces. When the velocity variable is not taken into account, the usual
Chemin-Lerner space was first introduced in  \cite{CL} to study the existence of solutions to the incompressible Navier-Stokes equations in $\R^3$. To the best of our knowledge, Theorem \ref{main.th.} is the first result for the application of such space to the well-posedness theory of the Cauchy problem on the Boltzmann equation. Moreover, noticing for $s=3/2$ that the Besov space $B_x^{s}=B^{s}_{2,1}\hookrightarrow L^{\infty}_x$ but the Sobolev space  $H^{s}_x$ is  not embedded into $L^\infty_x$, the regularity with respect to $x$ variable   that we consider here is critical. In most of previous work \cite{AMUXY-ARMA,D-08,GR,Guo-BE-s,Liu-Yang-Yu,Sh}, the Sobolev space $H^s_x$ $(s>3/2)$ obeying the Banach algebra property is typically used  for global well-posedness  of strong or classical solutionss. Therefore, Theorem   \ref{main.th.} presents a global result not only in a larger class of function spaces but also in a space with spatially critical regularity in the sense that we mentioned.

For readers we would also point out another two kinds of applications of the Besov space to the Boltzmann equation. In fact, in Ars\'enio-Masmoudi \cite{AM}, a new approach to velocity averaging lemmas in some Besov spaces  is developed basing on the dispersive property of the kinetic transport equation, and in Sohinger-Strain \cite{SS}, the optimal time decay rates in the whole space are investigated in the framework of \cite{GR} under the additional assumption that initial data belongs to a negative power Besov space $B^{s}_{2,\infty}$ for some $s<0$ with respect to $x$ variable.

In what follows let us a little detailedly recall some related works as far as  the choice of different function spaces for well-posedness of the Boltzmann equation near Maxwellians is concerned.
The first global
existence theorem for the mild solution is given by Ukai \cite{U74,U-s} in the space
\begin{equation}\notag
L^\infty \Big(0,\infty;L^\infty_{\be}(\R^3_\xi;H^N(\R^3_x))\Big),\ \ \be>\frac{5}{2},\ N\geq
2,
\end{equation}
by using the spectrum method as well as the contraction mapping principle, see also Nishida-Imai \cite{NI} and Kawashima \cite{Ka-BE13}. Here $L_{\be}^\infty(\R^3_\xi)$ denotes a space of all functions $f$ with  $ (1+|\xi|)^{\be} f$ uniformly bounded. Using a similar approach, Shizuta \cite{Sh} obtains the global existence of the classical solution $f(t,x,\xi)\in C^{1,1,0}((0,\infty)\times \T^3_x\times \R^3_\xi )$  on torus, with the uniform bound in the space
\begin{equation}\notag
    L^\infty \Big(0,\infty; L^\infty_{\be} (\R^3_\xi; C^s(\T^3_x))\Big),\ \ {\be}>\frac{5}{2}, \
    s>\frac{3}{2}.
\end{equation}
The spectrum method was later improved in Ukai-Yang \cite{UY-AA} for the existence of the mild solution in the space
\begin{equation}\notag
   L^\infty \Big(0,\infty;    L^2(\R^3_x\times\R^3_\xi)\cap
    L^\infty_{\be}(\R^3_\xi;L^\infty(\R^3_x))\Big),\ \
    \be>\frac{3}{2},
\end{equation}
without any regularity conditions, where some $L^\infty$-$L^2$  estimates in terms of the Duhamel's principle are developed.

On the
other hand, by means of the robust energy method, for instance, Guo \cite{Guo-IUMJ}, Liu-Yang-Yu \cite{
Liu-Yang-Yu} and Liu-Yu \cite{LY-S}, the well-posedness of classical solutions is also established in
the space
\begin{equation}\notag
 C\Big(0,\infty; H^{N}_{t,x,\xi}(\R^3_x\times\R^3_\xi)\Big), \ \ N\geq 4,
\end{equation}
where the Sobolev space  $H^{N}_{t,x,\xi}(\R^3_x\times\R^3_\xi)$ denotes a set of all functions whose derivatives with respect to all
variables $t$, $x$ and $\xi$ up to $N$ order are integrable in
$L^2(\R^3_x\times\R^3_\xi)$. It turns out that if only the strong solution with the uniqueness property is considered  then the time differentiation can be disregarded in the above  Sobolev space. Indeed the first author of the paper  studied in \cite{D-08} the existence of such strong solution in the space
$$
 C\Big(0,\infty; L^2(\R^3_\xi;H^N(\R^3_x))\Big),\ \ N\geq 4,
$$
where $N\geq 4$ actually can be straightforwardly extent to $N\geq 2$, cf.~\cite{GR}.  We here remark that the techniques used in \cite{D-08} lead to an extensive application of the Fourier energy method to the linearized Boltzmann equation as well as the related collision kinetic equations in plasma physics (cf.~\cite{D-11,DS-VPB}), in order to provide the time-decay properties of the linearized  solution operator instead of using the spectrum approach, and further give the optimal time-decay rates of solutions for the nonlinear problem.

The previously mentioned works are mainly focused on the angular cutoff Boltzmann equation. Recently, AMUXY \cite{AMUXY-ARMA} and  Gressman-Strain \cite{GR} independently prove the global existence of small-amplitude solutions in the angular non-cutoff case for general hard and soft potentials. Particularly, the function space for the energy that  \cite{GR} used in the hard potential case can take the form of
$$
 L^\infty\Big(0,\infty; L^2(\R^3_\xi;H^N(\R^3_x))\Big),\ \ N\geq 2.
$$
Notice that the energy dissipation rate in the non-cutoff case becomes more complicated compared to the cutoff situation, see \eqref{coer}, and the key issue in  \cite{AMUXY-ARMA} and \cite{GR} is to provide a good characterization of the Dirichlet form of the linearized Boltzmann operator so as to control  the nonlinear dynamics.
Very recently,  AMUXY \cite{AMUXY-13} presents a result for local existence in a function space significantly larger than those used in the existing works, where in non cutoff case the index of Sobolev spaces for the solution is related to the parameter of the angular singularity, and in cutoff case the solution space may take
$$
 L^\infty\Big(0,T_0; L^2(\R^3_\xi;H^s(\R^3_x))\Big),\ \ s>\frac{3}{2},
$$
where $T_0>0$ is a finite time.

As mentioned before, whenever $s=3/2$, since $H^{s}_x$ is not a Banach algebra, it seems impossible to expect to obtain a global result in  $L^\infty (0,\infty; X)$ with $X=L^2_\xi H^s_x$. A natural idea is to replace $H^s_x$ by $B^s_x=B^s_{2,1}$ which is a Banach algebra and of the critical regularity. In fact, instead of directly using $L^\infty (0,\infty; B_x^sL^2_\xi)$ we will consider the Chemin-Lerner type function space $\widetilde{L}_T^{\infty}\widetilde{L}_\xi^{2}(B^s_x)$ which has stronger topology than $L^\infty (0,\infty; B_x^sL^2_\xi)$. Here, for $T\geq 0$, $f(t,x,\xi)\in \widetilde{L}_T^{\infty}\widetilde{L}_\xi^{2}(B^s_x)$ means that the norm
$$
\sum_{q\geq -1} 2^{qs} \sup_{0\leq t\leq T} \|\De_q f(t,\cdot,\cdot)\|_{L^2_{x,\xi}}
$$
is finite. The reason why the supremum with respect to time is put after the summation is that one has to use such stronger norm to control the nonlinear term, for instance, see the estimate on $I_1$ given by \eqref{I1}.

In what follows we explain the technical part in the proof of Theorem  \ref{main.th.}.
First, compared to the case of the fluid dynamic equations mentioned before, the corresponding estimates in the space $\widetilde{L}_T^{\infty}\widetilde{L}_\xi^{2}(B^s_x)$ for the Boltzmann equation is much more complicated, not only because of
the additional velocity variable $\xi$ but also because of the nonlinear integral operator $\Gamma(f,f)$. In fact, applying the energy estimates in $L^2_{x,\xi}$ to \eqref{f.eq.} for each $\De_qf$ and due to the choice of the solution space $\widetilde{L}_T^{\infty}\widetilde{L}_\xi^{2}(B^s_x)$, one has to take the time integral, take the square root on both sides of the resulting estimate and then take the summation over all $q\geq -1$, so that we are formally forced to make the trilinear estimate in the form of
$$
\sum\limits_{q\geq-1}2^{qs}\left[\int_0^T\left|(\Delta_q\Gamma(f,g),\Delta_q h)\right|dt\right]^{1/2}.
$$
We need to control the above trilinear term in terms of the norms $\CE_T(\cdot)$ and $\CD_T(\cdot)$ which correspond to the linearized dynamics in the same space $\widetilde{L}_T^{\infty}\widetilde{L}_\xi^{2}(B^s_x)$  and play the usual role of the energy and the energy dissipation respectively.  Thus, Lemma \ref{basic.non.es.} becomes the key step in the proof of the global existence. Due to the nonlinearity of $\Ga(f,g)$, we need to use the Bony decomposition, for instance, for the loss term,
$$
f_{\ast}g=\mathcal {T}_{f_{\ast}}g+\mathcal {T}_{g}f_{\ast}+\mathcal {R}(f_{\ast},g),
$$
where notions $\mathcal {T}$ and $\mathcal {R}$ are to be given later on. For each term on the right of the above decomposition, we will use the boundedness  property in $L^p$ $(1\leq p\leq \infty)$ for the operators $\De_q$ and $S_q$, see  \eqref{bdop} in the Appendix, and further make use of the techniques in \cite{CL} developed by Chemin and Lerner to close the estimate on the above trilinear term.

Second, it seems very standard (cf.~\cite{D-08,DS-VPB}) to estimate the macroscopic component $(a,b,c)$ in the space $\widetilde{L}_T^2(B_x^{1/2})$ on the basis of the fluid-type system \eqref{mac.law}. The only new difficulty that we have to overcome lies on the estimates on
$$
\sum_{q\geq -1} 2^{qs}\left(\left\|\Lambda_i\left(\Delta_q\mathbbm{h}\right)\right\|_{L_T^2L^2_x}+\left\|\Theta_{im}\left(\Delta_q\mathbbm{h}\right)\right\|_{L_T^2L^2_x}\right),\quad s=1/2,
$$
where $\Lambda_i$ and $\Theta_{im}$ with $1\leq i,m\leq 3$ are velocity moment functions, defined later on, and $\mathbbm{h}=-L\{\FI-\FP\}f+\Ga(f,f)$. As in \cite{Guo-BE-s,Guo-L}, this actually can be done in the general situation given by Lemma \ref{non.EDL2.}. We remark that the proof of Lemma  \ref{non.EDL2.} is based on the key Lemma \ref{basic.non.es.} related to the trilinear estimate. The global a priori estimate \eqref{basic.eng.} then can be obtained by combining those trilinear estimates and the estimate on  the macroscopic dissipations.

Third, since the local existence in the space $\widetilde{L}_T^{\infty}\widetilde{L}_\xi^{2}(B^s_x)$  is not obvious, we also provide the complete proof of that by using the idea of \cite{SG}, for instance. The main goal in this part is to obtain the uniform bound of an approximate solution sequence in the norm $\widetilde{Y}_T(\cdot)$ given in \eqref{fn.bdd}, and also show the uniqueness and continuity with respect to $T$ of solutions satisfying such uniform bound. Here, once again the result of Lemma \ref{basic.non.es.} is needed. We also point out that the proof of continuity for $T\mapsto \widetilde{Y}_T(f)$ is essentially reduced to prove that
$$
t\mapsto \sum_{q\geq -1} 2^{3q/2} \|\De_q f(t,\cdot,\cdot )\|_{L^2_{x,\xi}},
$$
is continuous, with the help of the boundedness of the norm $\widetilde{\CD}_T(f)$, see Theorem \ref{local.existence} for more details.



The rest of the paper is arranged as follows.
In Section 2, we explain some notations and present definitions of some function spaces. In Sections 3 and 4, we deduce the key estimates for the collision operators $\Gamma$
and $L$. The estimate for the macroscopic dissipation is given in Section 5. Section 6 is devoted to obtaining the global a priori estimates for the Boltzmann equation in the Chemin-Lerner type space. In Section 7, we construct the local solutions of the Bolztmann equation and further show qualitative properties of the constructed local solutions. In Section 8 we give the proof of Theorem \ref{main.th.}. Finally, an appendix is given for some preliminary lemmas which will be used in the previous sections.

\section{Notations and function spaces}

Throughout the paper,  $C$ denotes some generic positive (generally large) constant and $\lambda$ denotes some generic positive (generally small) constant, where both $C$ and $\lambda$ may take different values in different places. For two quantities $A$ and $B$, ${A}\lesssim {B}$ means that  there is a generic constant $C>0$
such that $ {A}\leqslant C {B}$, and  $ {A}\sim  {B}$ means ${A}\lesssim {B}$ and $ {B}\lesssim {A}$.  For simplicity, $(\cdot,\cdot)$ stands for
the inner product in either $L^2_x=L^2(\R^3_x)$, $L^2_\xi=L^2(\R^3_\xi)$ or $L^2_{x,\xi}=L^2(\R^3_x\times \R^3_\xi)$. We use $\mathcal {S}(\R^3)$ to denote the Schwartz function space on $\R^3$,
and use $\mathcal {S}'(\R^3)$ to denote the dual space of $\mathcal {S}(\R^3)$, that is the tempered function space.
%

Since the crucial nonlinear estimates require a
dyadic decomposition of the Fourier variable, in what follows we recall briefly the
Littlewood-Paley decomposition theory and some function spaces, such as
the Besov space and the Chemin-Lerner space. Readers may refer to \cite{BCD} for more details.
Let us start with the Fourier transform. Here and below, the Fourier transform is taken with respect to the variable $x$ only,
not variables $\xi$ and $t$. Given $(t,\xi)$, the Fourier transform $\hat{f}(t,k,\xi)=\CF_x f(t,k,\xi)$ of a Schwartz function $f(t,x,\xi)\in \CS(\R^3_x)$ is given by
$$
\hat{f}(t,k,\xi):=\int_{\R^3}dx\, e^{-i x\cdot k}f(t,x,\xi),
$$
and  the Fourier transform of a tempered function $f(t,x,\xi)\in \CS'(\R^3_x)$ is defined by the dual argument in the standard way.
%

We now introduce a dyadic partition of $\mathbb{R}^{3}_x$. Let $(\varphi, \chi)$ be a couple of smooth functions valued in the closed interval $[0,
1]$ such that $\varphi$ is supported in the shell
$\mathbb{C}(0,\frac{3}{4},\frac{8}{3})
=\{k\in\mathbb{R}^{3}:\frac{3}{4}\leq|k|\leq\frac{8}{3}\}$ and
$\chi$ is supported in the ball $\B(0,\frac{4}{3})=
\{k\in\mathbb{R}^{3}:|k|\leq\frac{4}{3}\}$, with
\begin{equation*}
\begin{split}
\chi(k)+\sum_{q\geq0}\varphi(2^{-q}k)=1,\ \
\forall\,k\in\mathbb{R}^{3},\\
\sum_{q\in \mathbb{Z}}\varphi(2^{-q}k)=1,\ \
\forall\,k\in\mathbb{R}^{3}\backslash\{0\}.
\end{split}
\end{equation*}
The nonhomogeneous dyadic blocks of  $f=f(x)\in\mathcal{S'}(\R^3_x)$ are
defined as follows:
$$
\Delta_{-1}f:=\chi(D)f=\tilde{\psi}\ast f=\int_{\R^3} \tilde{\psi}(y)f(x-y)\,dy, \ \ \ \mbox{with}\ \
\tilde{\psi}=\mathcal{F}^{-1}\chi;
$$
$$
\Delta_{q}f:=\varphi(2^{-q}D)f=2^{3q}\int_{\R^3} \psi(2^{q}y)f(x-y)\,dy\ \
\ \mbox{with}\ \ \psi=\mathcal{F}^{-1}\varphi,\quad
q\geq0,
$$
where $\ast$  is the convolution operator with respect to the variable $x$ and $\mathcal{F}^{-1}$ denotes the
inverse Fourier transform.
Define the low frequency cut-off operator $S_q$ {($q\geq-1$)} by
$$S_{q}f:=\sum_{j\leq q-1}\Delta_{j}f.$$
It is a convention that $S_{0}f=\Delta_{-1}f$ for $q=0$, and $S_{-1}f=0$ in the case of $q=-1$.
Moreover, the homogeneous dyadic blocks are defined by
$$
\dot{\Delta}_{q}f:=\varphi(2^{-q}D)f=2^{3q}\int_{\R^3} \psi(2^{q}y)f(x-y)dy,\quad \forall\,q\in\mathbb{Z}.
$$
With these notions, the nonhomogeneous Littlewood-Paley
decomposition of  $f\in
\mathcal{S'}(\R^3_x)$ is given by
$$
f=\sum_{q \geq-1}\Delta_{q}f.
$$
And for $f\in \mathcal {S}'$, one  also has
$$f=\sum\limits_{q \in\mathbb{Z}}\dot{\Delta}_{q}f
$$
modulo a polynomial only. Recall that the above
Littlewood-Paley decomposition is almost orthogonal in $L_x^2$.

Having defined the linear operators $\Delta_q$ for $q\geq -1$ (or $\dot{\Delta}_q$ for $q\in\mathbb{Z})$, we give
the definition of nonhomogeneous (or homogeneous) Besov spaces as follows.

\begin{definition}
Let $1\leq p\leq\infty$ and $s\in \mathbb{R}$. For $1\leq r\leq \infty$, the nonhomogeneous Besov space
$B^{s}_{p,r}$ is defined by
\begin{multline*}
B^{s}_{p,r}:=\{ f\in \mathcal {S}'(\R^3_x):\ f=\sum_{q \geq-1}\Delta_{q}f\ \mbox{in}\
\mathcal{S'},\ \text{with} \\
\|f\|_{B^s_{p, r}}=:
\left(\sum_{q\geq-1}\left(2^{qs}\|\Delta_{q}f\|_{L_x^{p}}\right)^{r}\right)^{\frac{1}{r}}<\infty\},
\end{multline*}
where in the case $r=\infty$ we set
$$
\|f\|_{B^s_{p, \infty}}=
\sup_{q\geq-1}2^{qs}\|\Delta_{q}f\|_{L_x^{p}}.
$$
\end{definition}

Let $\mathcal {P}$ denote the class of all polynomials on $\R^3_x$ and let $\mathcal {S}'
/\mathcal {P}$ denote the tempered
distributions on $\R^3_x$ modulo polynomials. The corresponding definition for the homogeneous Besov space
is given as follows.

\begin{definition}
Let $1\leq p\leq\infty$ and $s\in \mathbb{R}$. For $1\leq r<\infty$, the homogeneous
Besov space is defined by
\begin{multline*}
\dot{B}^{s}_{p,r}:=\{ f\in\mathcal {S}'
/\mathcal {P}:\ f=\sum_{q \in\mathbb{Z}}\Delta_{q}f\ \mbox{in}
\ \mathcal{S'}/\mathcal {P},\ \text{with}\\
\|f\|_{\dot{B}^s_{p, r}}:=
\left(\sum\limits_{q\in\mathbb{Z}}(2^{qs}\|{\dot{\Delta}_{q}}f\|_{L_x^{p}})^{r}\right)^{\frac{1}{r}}<\infty\},
\end{multline*}
where in the case $r=\infty$ we set
$$
\|f\|_{\dot{B}^s_{p, \infty}}=
\sup_{q\in \mathbb{Z}}2^{qs}\|{\dot{\Delta}_{q}}f\|_{L_x^{p}}<\infty.
$$
\end{definition}

To the end, for brevity of presentations, we denote $B^s_{2,1}$ by $B^s$, and $\dot{B}^s_{2,1}$ by $\dot{B}^s$, respectively, and we also write $B^s$, $\dot{B}^s$ as $B^s_x$, $\dot{B}^s_x$ to emphasize the $x$ variable.

Since the velocity distribution function $f=f(t,x,\xi)$ involves the velocity variable $\xi$ and the time variable $t$, it is natural to define the Banach space valued function space
$$
L^{p_1}_TL^{p_2}_\xi L^{p_3}_x:=L^{p_1}(0,T; L^{p_2}(\R^3_\xi; L^{p_3}(\R^3_x))),
$$
for $0<T\leq \infty$, $1\leq p_1,\, p_2,\, p_3\leq\infty$, with the norm
\begin{equation*}
\begin{split}
\|f\|_{L^{p_1}_TL^{p_2}_\xi L^{p_3}_x}
=&\left(\int_{0}^T\left(
\int_{\R^3}\left(\int_{\R^3}|f(t,x,\xi)|^{p_3}dx\right)^{p_2/p_3}d\xi\right)
^{p_1/p_2}dt \right)^{1/p_1},
\end{split}
\end{equation*}
where we have used the normal convention in the case when $p_1=\infty$, $p_2=\infty$ or $p_3=\infty$.
Moreover, in order to characterize the Boltzmann dissipation rate, we also define the following velocity
weighted norm
\begin{equation*}
\begin{split}
\|f\|_{L_T^{p_1}L_{\xi,\nu}^{p_2}L_x^{p_3}}
=&\left(\int_{0}^T\left(
\int_{\R^3}\nu(\xi)\left(\int_{\R^3}|f(t,x,\xi)|^{p_3}dx\right)^{p_2/p_3}d\xi\right)
^{p_1/p_2}d t\right)^{1/p_1}
\end{split}
\end{equation*}
for $0<T\leq \infty$, $1\leq p_1,\, p_2,\, p_3\leq\infty$, where the normal convention in the case when $p_1=\infty$, $p_2=\infty$ or $p_3=\infty$ has been used.
%


In what follows, we present the definition of the Chemin-Lerner type spaces
$$
\widetilde{L}_T^{\varrho_1}\widetilde{L}_\xi^{\varrho_2}(B_{p,r}^s),\quad \widetilde{L}_T^{\varrho_1}\widetilde{L}_{\xi,\nu}^{\varrho_2}(B_{p,r}^s),
$$
and
$$
\widetilde{L}_T^{\varrho_1}\widetilde{L}_\xi^{\varrho_2}(\dot{B}_{p,r}^s), \quad \widetilde{L}_T^{\varrho_1}\widetilde{L}_{\xi,\nu}^{\varrho_2}(\dot{B}_{p,r}^s),
$$
{which are initiated by the work \cite{CL}}.

\begin{definition}
Let $1\leq \varrho_1, \varrho_2, p, r\leq\infty$ and $s\in \mathbb{R}$. For $0< T\leq \infty$, the space $\widetilde{L}_t^{\varrho_1}\widetilde{L}_\xi^{\varrho_2}\Red{(B_{p,r}^s)}$ is defined by
$$
\widetilde{L}_T^{\varrho_1}\widetilde{L}_\xi^{\varrho_2}(B_{p,r}^s)
=\left\{
f(t,\cdot,\xi)\in \CS':\ \|f\|_{\widetilde{L}_T^{\varrho_1}\widetilde{L}_\xi^{\varrho_2}(B_{p,r}^s)}<\infty\right\},
$$
where
\begin{equation*}
\|f\|_{\widetilde{L}_T^{\varrho_1}\widetilde{L}_\xi^{\varrho_2}(B_{p,r}^s)}
=\left(\sum_{q\geq-1}2^{qsr}\left(\int_{0}^T\left(
\int_{\R^3}\|\Delta_{q}f\|^{\varrho_2}_{L_x^{p}}d\xi\right)
^{\varrho_1/\varrho_2}dt\right)^{r/\varrho_1}\right)^{\frac{1}{r}}
\end{equation*}
that is,
\begin{equation*}
\|f\|_{\widetilde{L}_T^{\varrho_1}\widetilde{L}_\xi^{\varrho_2}(B_{p,r}^s)}
=\left(\sum_{q\geq-1}2^{qsr}\|\Delta_{q}f\|_{{L_T^{\varrho_1}}L_\xi^{\varrho_2}L_x^{p}}^{r}\right)^{1/r},
\end{equation*}
with the usual convention for $\varrho_1, \varrho_2, p, r=\infty$. Similarly, one also denotes
\begin{equation*}
\begin{split}
\|f\|_{\widetilde{L}_T^{\varrho_1}\widetilde{L}_{\xi,\nu}^{\varrho_2}(B_{p,r}^s)}
=\left(\sum_{q\geq-1}2^{qsr}\|\Delta_{q}f\|_{{L_T^{\varrho_1}L_{\xi,\nu}^{\varrho_2}L_x^{p}}}^{r}\right)^{1/r},
\end{split}
\end{equation*}
and
\begin{equation*}
\begin{split}
&\|f\|_{\widetilde{L}_T^{\varrho_1}\widetilde{L}_{\xi}^{\varrho_2}(\dot{B}_{p,r}^s)}
=\left(\sum_{q\in \mathbb{Z}}2^{qsr}\|{\dot{\Delta}_{q}}f\|_{{L_T^{\varrho_1}L_{\xi}^{\varrho_2}L_x^{p}}}^{r}\right)^{1/r},\\
&\|f\|_{\widetilde{L}_T^{\varrho_1}\widetilde{L}_{\xi,\nu}^{\varrho_2}(\dot{B}_{p,r}^s)}
=\left(\sum_{q\in \mathbb{Z}}2^{qsr}\|{\dot{\Delta}_{q}}f\|_{{L_T^{\varrho_1}L_{\xi,\nu}^{\varrho_2}L_x^{p}}}^{r}\right)^{1/r},
\end{split}
\end{equation*}
with the usual convention for $\varrho_1, \varrho_2, p, r=\infty$.
\end{definition}

We conclude this section with a few remarks. First, since the goal of the paper is to establish the well-posedness in the spatially critical Besov space for the Boltzmann equation, we mainly consider the above norms in the case that $p=2$, $r=1$ and $\varrho_2=2$. Thus, the spaces
$\widetilde{L}_T^{\varrho}\widetilde{L}_\xi^{2}(B^s_x)$, $\widetilde{L}_T^{\varrho}\widetilde{L}_{\xi,\nu}^{2}(B^s_x)$, $\widetilde{L}_T^{\varrho}\widetilde{L}_\xi^{2}(\dot{B}^s_x)$ and $\widetilde{L}_T^{\varrho}\widetilde{L}_{\xi,\nu}^{2}(\dot{B}^s_x)$ with $\varrho=0$ or $\infty$ will be frequently used. Next, whenever a function $f=f(t,x,\xi)$ is independent of $t$ or $\xi$, the corresponding norms defined above are modified in the usual way by omitting
the $t-$variable or $\xi-$variable, respectively. Finally, it should be pointed out that {the Chemin-Lerner type norm $\|\cdot\|_{\widetilde{L}_T^{\varrho_1}\widetilde{L}_\xi^{\varrho_2}(B_{p,r}^s)}$ is the refinement of the usual norm $\|\cdot\|_{L_T^{\varrho_1}L_\xi^{\varrho_2}(B^{s}_{p,r})}$ given by}
\begin{equation*}
\begin{split}
\|f\|_{L_T^{\varrho_1}L_\xi^{\varrho_2}(B^{s}_{p,r})}
=&\left(\int_{0}^T\left(
\int_{\R^3}\left(\sum_{q\geq-1}2^{qsr}\|\Delta_{q}f\|_{L_x^{p}}^{r}\right)^{\varrho_2/r}d\xi\right)
^{\varrho_1/\varrho_2}dt\right)^{1/\varrho_1},
\end{split}
\end{equation*}
for $0<T\leq \infty$, $1\leq r, p, \varrho_1, \varrho_2\leq\infty$.

\section{Trilinear estimates}


Recall the Boltzmann equation \eqref{f.eq.}. The linearized collision operator $L$ can be written as
$
L=\nu-K.
$
Here the multiplier $\nu=\nu(\xi)$, called the collision frequency,  is given by
\begin{equation}\notag
\nu(\xi)=\int_{\R^3} d\xi_\ast \int_{\S^2} d \omega \, |\xi-\xi_\ast|^{\gamma}B_0(\theta)\mu(\xi_\ast).
\end{equation}
It holds that $\nu(\xi)\thicksim (1+|\xi|)^\ga$, cf. \cite{CIP-Book,UY-AA}.
And the integral operator $K=K_{2}-K_{1}$ is defined
as
\begin{eqnarray}
[K_{1}f](\xi)&=&\int_{\R^3} d\xi_\ast \int_{\S^2} d \omega \, |\xi-\xi_\ast|^{\gamma}B_0(\theta)\mu^{1/2}(\xi_\ast)\mu^{1/2}(\xi)f(\xi_\ast),\label{def.k1}\\[2mm]
[K_{2}f](\xi)&=&\int_{\R^3} d\xi_\ast \int_{\S^2} d \omega \, |\xi-\xi_\ast|^{\gamma}B_0(\theta)\mu^{1/2}(\xi_\ast)\notag\\
&&\qquad\qquad\qquad\qquad \times\left\{\mu^{1/2}(\xi_\ast')f(\xi')
+\mu^{1/2}(\xi')f(\xi_\ast')\right\}.\label{K.def.}
\end{eqnarray}
$L$ is coercive in the sense that there is $\la_0>0$ such that
\begin{equation}
\label{coer}
\int_{\R^3} f Lf \,d\xi\geq \la_0 \int_{\R^3} \nu(\xi) |\{\FI-\FP\} f|^2.
\end{equation}
Moreover, the nonlinear  collision operator $\Gamma (f, g)$ is written as
\begin{equation}\label{nop.def.}
\begin{split}
\Gamma(f,g)=&\mu^{-1/2}(\xi)Q\left[\mu^{-1/2}f,\mu^{-1/2}g\right]
=\Gamma_{gain}(f,g)-\Gamma_{loss}(f,g)\\[2mm]
=&\int_{\R^3} d\xi_\ast \int_{\S^2} d \omega \, |\xi-\xi_\ast|^{\gamma}B_0(\theta)\mu^{1/2}(\xi_\ast)f(\xi_\ast')g(\xi')\\[2mm]
&-g(\xi)\int_{\R^3} d\xi_\ast \int_{\S^2} d \omega \, |\xi-\xi_\ast|^{\gamma}B_0(\theta)\mu^{1/2}(\xi_\ast)f(\xi_\ast).
\end{split}
\end{equation}

In this section, we intend to give the key estimates for the nonlinear Bolztmann collision operator
$\Gamma(\cdot,\cdot)$ defined by \eqref{nop.def.} in terms of the spatially critical Besov space. It should be pointed that the following lemmas are new and they play a crucial role in the proof of
the global existence of solutions to the Boltzmann equation \eqref{f.eq.}. First we show the trilinear estimate in the following

\begin{lemma}\label{basic.non.es.}
Assume $s>0$, $0<T\leq \infty$. Let $f=f(t,x,\xi)$,  $g=g(t,x,\xi)$, and $h=h(t,x,\xi)$ be three suitably smooth distribution functions such that all the  norms on the right  of the following inequalities are well defined, then it holds that
\begin{multline}
\sum\limits_{q\geq-1}2^{qs}\left[\int_0^T\left|(\Delta_q\Gamma(f,g),\Delta_q h)\right|dt\right]^{1/2}
\lesssim \|h\|^{1/2}_{\widetilde{L}^{2}_T\widetilde{L}^{2}_{\xi,\nu} (B^{s}_x)}\\
\times \Bigg[
\|g\|^{1/2}_{\widetilde{L}^{2}_T\widetilde{L}^{2}_{\xi,\nu} (B^{s}_x)}
\|f\|^{1/2}_{L^{\infty}_T L^{2}_\xi L^\infty_x}+\|f\|^{1/2}_{L^{2}_T L^{2}_{\xi,\nu} L^\infty_x}
\|g\|^{1/2}_{\widetilde{L}^{\infty}_T\widetilde{L}^{2}_\xi (B^{s}_x)}
\\[3mm]
+\|f\|^{1/2}_{\widetilde{L}^{2}_T\widetilde{L}^{2}_{\xi,\nu} (B^{s}_x)}
\|g\|^{1/2}_{L^{\infty}_T L^{2}_\xi L^\infty_x}+\|g\|^{1/2}_{L^{2}_T L^{2}_{\xi,\nu} L^\infty_x}
\|f\|^{1/2}_{\widetilde{L}^{\infty}_T\widetilde{L}^{2}_\xi (B^{s}_x)}\Bigg],
\label{basic.non.p1.}
\end{multline}
%
where the inner product $(\cdot,\cdot)$ is taken with respect to variables $(x,\xi)$.
\end{lemma}

\begin{proof}
Recalling \eqref{nop.def.} and using the inequality $(A+B)^{1/2}\leq A^{1/2}+B^{1/2}$ for $A\geq 0$ and $B\geq 0$, one has
\begin{eqnarray}
&\dis \left[\int_0^T\left|(\Delta_q\Gamma(f,g),\Delta_q h)\right|dt\right]^{\frac{1}{2}}\notag\\
&\dis \leq \left[\int_0^T\left|(\Delta_q\Gamma_{gain}(f,g),\Delta_q h)\right|dt\right]^{\frac{1}{2}}+\left[\int_0^T\left|(\Delta_q\Gamma_{loss}(f,g),\Delta_q h)\right|dt\right]^{\frac{1}{2}}.
\label{Da.p01}
\end{eqnarray}
Here notice that since the collision integral acts on $\xi$ variable only and $\De_q$ acts on $x$ variable only, one can write
\begin{eqnarray*}
\Delta_q\Gamma_{gain}(f,g) &=&\int_{\R^3} d\xi_\ast \int_{\S^2} d \omega \, |\xi-\xi_\ast|^{\gamma}B_0(\theta)\mu^{1/2}(\xi_\ast)\De_q [f(\xi_\ast')g(\xi')],\\
\De_q\Gamma_{loss}(f,g) &=&\int_{\R^3} d\xi_\ast \int_{\S^2} d \omega \, |\xi-\xi_\ast|^{\gamma}B_0(\theta)\mu^{1/2}(\xi_\ast)\De_q [f(\xi_\ast)g(\xi)].
\end{eqnarray*}
By applying Cauchy-Schwarz inequality to both integrals on the right of \eqref{Da.p01} with respect to all variable $(t,x,\xi,\xi_\ast,\om)$,  making the change of variables $(\xi,\xi_\ast)\to(\xi',\xi'_\ast)$ in the gain term, and then taking the summation over $q\geq -1$ after multiplying it by $2^{qs}$,  we see that
\begin{equation}\notag
\begin{split}
&\sum\limits_{q\geq-1}2^{qs}\left[\int_0^T\left|(\Delta_q\Gamma(f,g),\Delta_qh)\right|dt\right]^{1/2}\\
&\lesssim \sum\limits_{q\geq-1}2^{qs}\left[\left(\int_0^Tdt\int_{\R^9\times\S^2}dxd\xi d\xi_{\ast}d\omega\,
|\xi'-\xi_{\ast}'|^\ga\mu^{1/2}(\xi'_\ast)\left|\Delta_q[f_{\ast}g]\right|^2\right)^{1/2}\right]^{1/2}
\\&\qquad \times\left[\left(\int_0^Tdt\int_{\R^9}dxd\xi d\xi_{\ast}d\omega\,
|\xi-\xi_{\ast}|^\ga\mu^{1/2}(\xi_\ast)\left|\Delta_q h\right|^2\right)^{1/2}\right]^{1/2}\\
&\quad +\sum\limits_{q\geq-1}2^{qs}\left[\left(\int_0^Tdt\int_{\R^9\times \S^2}dxd\xi d\xi_{\ast}d\omega\,
|\xi-\xi_{\ast}|^\ga\mu^{1/2}(\xi_\ast)\left|\Delta_q[f_{\ast}g]\right|^2\right)^{1/2}\right]^{1/2}
\\&\qquad \times\left[\left(\int_0^Tdt\int_{\R^9}dxd\xi d\xi_{\ast}d\omega\,
|\xi-\xi_{\ast}|^\ga\mu^{1/2}(\xi_\ast)\left|\Delta_q h\right|^2\right)^{1/2}\right]^{1/2}\\[3mm]
&:=I_{0},
\end{split}
\end{equation}
where $0\leq B_0(\theta)\leq C |\cos\theta| \leq C$ have been used.
Further by using the discrete version of Cauchy-Schwarz inequality to two summations $\sum_{q\geq -1}$ above, one obtains that
\begin{equation}\notag
\begin{split}
I_0\lesssim& \left[\sum\limits_{q\geq-1}2^{qs}\left(\int_0^Tdt\int_{\R^9}dxd\xi d\xi_{\ast}\,
|\xi-\xi_{\ast}|^\ga\left|\Delta_q[f_{\ast}g]\right|^2\right)^{1/2}\right]^{1/2}
\\&\times\left[\sum\limits_{q\geq-1}2^{qs}\left(\int_0^Tdt\int_{\R^9}dxd\xi d\xi_{\ast}\,
|\xi-\xi_{\ast}|^\ga\mu^{1/2}(\xi_\ast)\left|\Delta_q h\right|^2\right)^{1/2}\right]^{1/2}\\[3mm]
&:=I^{1/2}\times II^{1/2},
\end{split}
\end{equation}
where $|\xi'-\xi_\ast'|=|\xi-\xi_\ast|$, {$\mu^{1/2}(\xi'_\ast)\leq1$} and $\int_{\S^2}d\om=4\pi$ have been used.
It is straightforward to see
$$
II\leq \|h\|_{\widetilde{L}^{2}_T\widetilde{L}^{2}_{\xi,\nu} (B^{s}_x)}
$$
due to
\begin{equation}
\notag
\int_{\R^3}d\xi_\ast\, |\xi-\xi_{\ast}|^\ga\mu^{1/2}(\xi_\ast) \sim  (1+|\xi|)^\ga \sim \nu(\xi).
\end{equation}
We now turn to compute $I$.
Recalling the Bony's decomposition, one can write $\Delta_q[f_{\ast}g]$ as
$$
\Delta_q[f_{\ast}g]=\Delta_q\left[\mathcal {T}_{f_{\ast}}g+\mathcal {T}_{g}f_{\ast}+\mathcal {R}(f_{\ast},g)\right].
$$
Here
$\mathcal {T}_{\cdot}\cdot$, and $\mathcal {R}(\cdot,\cdot)$ are the usual paraproduct operators. They are defined as follows. For suitable smooth distribution functions $u$ and $v$,
$$
\mathcal {T}_{u}v=\sum\limits_{j}S_{j-1}u\Delta_jv,\ \ \mathcal {R}(u,v)=\sum\limits_{|j'-j|\leq1}\Delta_{j'}u\Delta_jv.
$$
We therefore get from Minkowski's inequality that
\begin{equation}\notag
\begin{split}
I\leq& \sum\limits_{q\geq-1}2^{qs}\left(\int_0^Tdt\int_{\R^9}dxd\xi d\xi_{\ast}
|\xi-\xi_{\ast}|^\ga\left|\sum\limits_{j}\Delta_q[S_{j-1}f_\ast\Delta_j g]\right|^2\right)^{1/2}\\
&+\sum\limits_{q\geq-1}2^{qs}\left(\int_0^Tdt\int_{\R^9}dxd\xi d\xi_{\ast}
|\xi-\xi_{\ast}|^\ga\left|\sum\limits_{j}\Delta_q[S_{j-1}g\Delta_j f_\ast]\right|^2\right)^{1/2}\\
&+\sum\limits_{q\geq-1}2^{qs}\left(\int_0^Tdt\int_{\R^9}dxd\xi d\xi_{\ast}
|\xi-\xi_{\ast}|^\ga\left|\sum\limits_{|j-j'|\leq1}\Delta_q[\Delta_jf_\ast\Delta_{j'} g]\right|^2\right)^{1/2}\\[3mm]
&:=I_1+I_2+I_3.
\end{split}
\end{equation}
Now we estimate $I_1$, $I_2$ and $I_3$ term by term.

\medskip
\noindent \underline{{\it Estimates on $I_1$:}}
Notice that
$$
\sum\limits_{j}\Delta_q[S_{j-1}f_\ast\Delta_j g]=\sum\limits_{|j-q|\leq4}\Delta_q[S_{j-1}f_\ast\Delta_j g].
$$
By Minkowski's inequality again, one can see that
\begin{equation}\label{I1}
\begin{split}
I_1\leq& \sum\limits_{q\geq-1}\sum\limits_{|j-q|\leq4}2^{qs}\left(\int_0^Tdt\int_{\R^9}dxd\xi d\xi_{\ast}\,
|\xi|^\ga\left|\Delta_q[S_{j-1}f_\ast\Delta_j g]\right|^2\right)^{1/2}\\
&+\sum\limits_{q\geq-1}\sum\limits_{|j-q|\leq4}2^{qs}\left(\int_0^Tdt\int_{\R^9}dxd\xi d\xi_{\ast}\,
|\xi_\ast|^\ga\left|\Delta_q[S_{j-1}f_\ast\Delta_j g]\right|^2\right)^{1/2}\\[3mm]
&:=I_{1,1}+I_{1,2}.
\end{split}
\end{equation}
Applying \eqref{bdop} in the appendix, one can deduce that
\begin{equation*}
\begin{split}
I_{1,1}\leq& \sum\limits_{q\geq-1}\sum\limits_{|j-q|\leq4}2^{qs}\left(\int_0^Tdt\int_{\R^3}
\|f_\ast\|^2_{L^\infty_x}d\xi_{\ast}\int_{\R^3}|\xi|^\ga\left\|\Delta_j g\right\|_{L^2_x}^2d\xi \right)^{1/2}\\
\leq& \sum\limits_{q\geq-1}\sum\limits_{|j-q|\leq4}2^{qs}\left(\sup\limits_{0\leq t\leq T}\int_{\R^3}
\|f_\ast\|^2_{L^\infty_x}d\xi_{\ast}\int_0^Tdt\int_{\R^3}|\xi|^\ga\left\|\Delta_j g\right\|_{L^2_x}^2d\xi \right)^{1/2}
\\
\leq& \sum\limits_{q\geq-1}\sum\limits_{|j-q|\leq4}{2^{qs}}\left(\int_0^Tdt\int_{\R^3}|\xi|^\ga\left\|\Delta_j g\right\|_{L^2_x}^2d\xi \right)^{1/2}
\|f\|_{L^{\infty}_T L^{2}_\xi L^\infty_x}
\\
\leq& \sum\limits_{q\geq-1}\sum\limits_{|j-q|\leq4}{2^{(q-j)s}}c_1(j)\|g\|_{\widetilde{L}^{2}_T\widetilde{L}^{2}_{\xi,\nu} (B^{s}_x)}
\|f\|_{L^{\infty}_T L^{2}_\xi L^\infty_x},
\end{split}
\end{equation*}
where $c_1(j)$ is defined as
\begin{equation}
\label{def.c1j}
c_1(j)=\frac{2^{js}\left(\dis{\int_0^T}dt\dis{\int_{\R^3}}|\xi|^\ga\left\|\Delta_j g\right\|_{L^2_x}^2d\xi \right)^{1/2}}{{\|g\|_{\widetilde{L}^{2}_T\widetilde{L}^{2}_{\xi,\nu} (B^{s}_x)}}},
\end{equation}
which satisfies
$
{\|c_1(j)\|_{\ell^1}\leq1.}
$
From the above estimate on $I_{1,1}$, using the following convolution inequality for series
\begin{multline}
\label{con.ine}
\sum\limits_{q\geq-1}\sum\limits_{|j-q|\leq4}{2^{(q-j)s}}c_1(j)=\sum\limits_{q\geq-1}\left[\left({\bf1}_{|j|\leq4}2^{js}\right)\ast c_1(j)\right](q)
\\
\leq\|{\bf1}_{|j|\leq4}2^{js}\|_{\ell^1}\|c_1(j)\|_{\ell^1}<+\infty,
\end{multline}
we further get that
\begin{equation}\label{I11}
\begin{split}
I_{1,1}\lesssim\|g\|_{\widetilde{L}^{2}_T\widetilde{L}^{2}_{\xi,\nu} (B^{s}_x)}
\|f\|_{L^{\infty}_T L^{2}_\xi L^\infty_x},
\end{split}
\end{equation}

The estimates for $I_{1,2}$ is slightly different from $I_{1,1}$.  In fact, we may compute it as
\begin{equation*}
\begin{split}
I_{1,2}\leq& \sum\limits_{q\geq-1}\sum\limits_{|j-q|\leq4}2^{qs}\left(\int_0^Tdt\int_{\R^3}|\xi_\ast|^\ga
\|f_\ast\|^2_{L^\infty_x}d\xi_{\ast}\int_{\R^3}\left\|\Delta_j g\right\|_{L^2_x}^2d\xi \right)^{1/2}\\
\leq& \sum\limits_{q\geq-1}\sum\limits_{|j-q|\leq4}2^{qs}\left(\sup\limits_{0\leq t\leq T}\int_{\R^3}\left\|\Delta_j g\right\|_{L^2_x}^2d\xi\int_0^Tdt\int_{\R^3}|\xi_\ast|^\ga
\|f_\ast\|^2_{L^\infty_x}d\xi_{\ast} \right)^{1/2}
\\
\leq& \sum\limits_{q\geq-1}\sum\limits_{|j-q|\leq4}{2^{qs}}\left(\sup\limits_{0\leq t\leq T}\int_{\R^3}\left\|\Delta_j g\right\|_{L^2_x}^2d\xi \right)^{1/2}
\|f\|_{L^{2}_TL^{2}_{\xi,\nu}L^\infty_x}
\\
\leq& \sum\limits_{q\geq-1}\sum\limits_{|j-q|\leq4}{2^{(q-j)s}}c_2(j)
\|g\|_{\widetilde{L}^{\infty}_T\widetilde{L}^{2}_\xi (B^{s}_x)}\|f\|_{L^{2}_TL^{2}_{\xi,\nu}L^\infty_x},
\end{split}
\end{equation*}
where $c_2(j)$ is defined as
\begin{equation*}
c_2(j)=\frac{2^{js}\left(\sup\limits_{0\leq t\leq T}\dis{\int_{\R^3}}\left\|\Delta_j g\right\|_{L^2_x}^2d\xi \right)^{1/2}}{\|g\|_{\widetilde{L}^{\infty}_T\widetilde{L}^{2}_\xi (B^{s}_x)}}.
\end{equation*}
Since $\|c_2(j)\|_{\ell^1}=1$, then in a similar way as for obtaining \eqref{I11}, we have
\begin{equation}\label{I12}
\begin{split}
I_{1,2}\lesssim\|f\|_{L^{2}_T L^{2}_{\xi,\nu}L^\infty_x}
\|g\|_{\widetilde{L}^{\infty}_T \widetilde{L}^{2}_\xi(B^{s})}.
\end{split}
\end{equation}
Now substituting \eqref{I11} and \eqref{I12} into \eqref{I1}, one has
\begin{equation}\notag
I_1\lesssim\|g\|_{\widetilde{L}^{2}_T \widetilde{L}^{2}_{\xi,\nu}(B^{s}_x)}
\|f\|_{L^{\infty}_T L^{2}_\xi L^\infty_x}+\|f\|_{L^{2}_T L^{2}_{\xi,\nu}L^\infty_x}
\|g\|_{\widetilde{L}^{\infty}_T \widetilde{L}^{2}_\xi (B^{s}_x)}.
\end{equation}
This gives the estimate for $I_1$.

\medskip

\noindent \underline{{\it Estimates on $I_2$:}} From (3.14), we have
\begin{equation}\label{I2}
\begin{split}
I_2\leq& \sum\limits_{q\geq-1}\sum\limits_{|j-q|\leq4}2^{qs}\left(\int_0^Tdt\int_{\R^9}dxd\xi d\xi_{\ast}\,
|\xi|^\ga\left|\Delta_q[S_{j-1}g\Delta_j f_\ast]\right|^2\right)^{1/2}\\
&+\sum\limits_{q\geq-1}\sum\limits_{|j-q|\leq4}2^{qs}\left(\int_0^Tdt\int_{\R^9}dxd\xi d\xi_{\ast}\,
|\xi_\ast|^\ga\left|\Delta_q[S_{j-1}g\Delta_j f_\ast]\right|^2\right)^{1/2}\\[3mm]
&:=I_{2,1}+I_{2,2}.
\end{split}
\end{equation}
As before, it follows from  \eqref{bdop} in the appendix that
\begin{equation}\label{I21}
\begin{split}
I_{2,1}\leq& \sum\limits_{q\geq-1}\sum\limits_{|j-q|\leq4}2^{qs}\left(\int_0^Tdt\int_{\R^3}
\|\Delta_{j}f_\ast\|^2_{L^2_x}d\xi_{\ast}\int_{\R^3}|\xi|^\ga\left\|g\right\|_{L^\infty_x}^2d\xi \right)^{1/2}\\
\leq& \sum\limits_{q\geq-1}\sum\limits_{|j-q|\leq4}2^{qs}\left(\sup\limits_{0\leq t\leq T}\int_{\R^3}
\|\Delta_{j}f_\ast\|^2_{L^2_x}d\xi_{\ast}\int_0^Tdt\int_{\R^3}|\xi|^\ga\left\|g\right\|_{L^\infty_x}^2d\xi \right)^{1/2}
\\
\leq& \sum\limits_{q\geq-1}\sum\limits_{|j-q|\leq4}2^{qs}\left(\sup\limits_{0\leq t\leq T}\int_{\R^3}\|\Delta_{j}f_\ast\|^2_{L^2_{x}}d\xi_{\ast}
\right)^{1/2}\|g\|_{L^{2}_TL^{2}_{\xi,\nu}L^\infty_x}
\\
\leq& \sum\limits_{q\geq-1}\sum\limits_{|j-q|\leq4}2^{(q-j)s}c_3(j)\|f\|_{\widetilde{L}^{\infty}_T \widetilde{L}^{2}_{\xi} ((B^{s}_x))}
\|g\|_{L^{2}_T L^{2}_{\xi,\nu} L^\infty_x}
\end{split}
\end{equation}
with
\begin{equation*}
c_3(j)=\frac{2^{js}\left(\sup\limits_{0\leq t\leq T}\dis{\int_{\R^3}\|\Delta_{j}f\|^2_{L^2_{x}}d\xi} \right)^{1/2}}{\|f\|_{\widetilde{L}^{\infty}_T \widetilde{L}^{2}_{\xi} ((B^{s}_x))}}.
\end{equation*}
Similarly, it holds that
\begin{equation}\label{I22}
\begin{split}
I_{2,2}\leq& \sum\limits_{q\geq-1}\sum\limits_{|j-q|\leq4}2^{qs}\left(\int_0^Tdt\int_{\R^3}|\xi_\ast|^\ga
\|\Delta_{j}f_\ast\|^2_{L^2_x}d\xi_{\ast}\int_{\R^3}\left\|g\right\|_{L^\infty_x}^2d\xi \right)^{1/2}\\
\leq& \sum\limits_{q\geq-1}\sum\limits_{|j-q|\leq4}2^{qs}\left(\sup\limits_{0\leq t\leq T}\int_{\R^3}\left\|g\right\|_{L^\infty_x}^2d\xi\int_0^Tdt\int_{\R^3}|\xi_\ast|^\ga
\|\Delta_{j}f_\ast\|^2_{L^2_x}d\xi_{\ast} \right)^{1/2}
\\
\leq& \sum\limits_{q\geq-1}\sum\limits_{|j-q|\leq4}2^{qs}
\left(\int_0^Tdt\int_{\R^3}|\xi_\ast|^\ga
\|\Delta_{j}f_\ast\|^2_{L^2_x}d\xi_{\ast}\right)^{1/2}\|g\|_{L^{\infty}_TL^{2}_{\xi}L^\infty_x}
\\
\leq& \sum\limits_{q\geq-1}\sum\limits_{|j-q|\leq4}2^{(q-j)s}c_4(j)
\|f\|_{\widetilde{L}^{2}_T \widetilde{L}^{2}_{\xi,\nu} ((B^{s}_x))}\|g\|_{L^{\infty}_TL^{2}_{\xi}L^\infty_x}
\end{split}
\end{equation}
with
$$
c_4(j)=\frac{2^{js}\left(\dis{\int_0^Tdt}\dis{\int_{\R^3}|\xi|^\ga
\|\Delta_{j}f\|^2_{L^2_x}d\xi}\right)^{1/2}}{\|f\|_{\widetilde{L}^{2}_T \widetilde{L}^{2}_{\xi,\nu} (B^{s}_x)}}.
$$
Noticing $\|(c_3)\|_{\ell^1}=1$ and ${\|c_4(j)\|_{\ell^1}\leq1}$, it follows from \eqref{I2}, \eqref{I21} and \eqref{I22} that
\begin{equation}\notag
I_{2}\lesssim \|f\|_{\widetilde{L}^{\infty}_T \widetilde{L}^{2}_{\xi} ((B^{s}_x))}
\|g\|_{L^{2}_T L^{2}_{\xi,\nu} L^\infty_x}+\|f\|_{\widetilde{L}^{2}_T \widetilde{L}^{2}_{\xi,\nu} ((B^{s}_x))}\|g\|_{L^{\infty}_TL^{2}_{\xi}L^\infty_x}.
\end{equation}

\medskip
\noindent \underline{{\it Estimates on $I_3$:}}
We start from the fact that
$$
\sum\limits_{j}\sum\limits_{|j-j'|\leq1}\Delta_q[\Delta_{j'}f_\ast\Delta_j g]=\sum\limits_{\max\{j,j'\}\geq q-2}\sum\limits_{|j-j'|\leq1}\Delta_q[\Delta_{j'}f_\ast\Delta_j g].
$$
With this, one can see that
\begin{equation*}
\begin{split}
I_3\leq& \sum\limits_{q\geq-1}\sum\limits_{\max\{j,j'\}\geq q-2}\sum\limits_{|j-j'|\leq1}2^{qs}\left(\int_0^Tdt\int_{\R^9}dxd\xi d\xi_{\ast}\,
|\xi|\left|\Delta_q[\Delta_{j'}f_\ast\Delta_j g]\right|^2\right)^{1/2}\\
&+\sum\limits_{q\geq-1}\sum\limits_{\max\{j,j'\}\geq q-2}\sum\limits_{|j-j'|\leq1}2^{qs}\left(\int_0^Tdt\int_{\R^9}dxd\xi d\xi_{\ast}\,
|\xi_\ast|\left|\Delta_q[\Delta_{j'}f_\ast\Delta_j g]\right|^2\right)^{1/2}\\[3mm]
&:=I_{3,1}+I_{3,2}.
\end{split}
\end{equation*}
As before, applying \eqref{bdop}, we get that
\begin{equation*}
\begin{split}
I_{3,1}\leq& \sum\limits_{q\geq-1}\sum\limits_{j\geq q-3}2^{qs}\left(\int_0^Tdt\int_{\R^3}
\|f_\ast\|^2_{L^\infty_x}d\xi_{\ast}\int_{\R^3}|\xi|\left\|\Delta_j g\right\|_{L^2_x}^2d\xi \right)^{1/2}\\
\leq& \sum\limits_{q\geq-1}\sum\limits_{j\geq q-3}2^{qs}\left(\sup\limits_{0\leq t\leq T}\int_{\R^3}
\|f_\ast\|^2_{L^\infty_x}d\xi_{\ast}\int_0^Tdt\int_{\R^3}|\xi|\left\|\Delta_j g\right\|_{L^2_x}^2d\xi \right)^{1/2}
\\
\leq& \sum\limits_{q\geq-1}\sum\limits_{j\geq q-3}2^{qs-js}2^{js}\left(\int_0^Tdt\int_{\R^3}|\xi|\left\|\Delta_j g\right\|_{L^2_x}^2d\xi \right)^{1/2}
\|f\|_{L^{\infty}_T L^{2}_\xi L^\infty_x}
\\
\leq& \sum\limits_{q\geq-1}\sum\limits_{j\geq q-3}2^{qs-js}c_1(j)\|g\|_{\widetilde{L}^{2}_T \widetilde{L}^{2}_{\xi,\nu} (B^{s}_x)}
\|f\|_{L^{\infty}_T L^{2}_\xi L^\infty_x}\\
\lesssim & \|g\|_{\widetilde{L}^{2}_T \widetilde{L}^{2}_{\xi,\nu}(B^{s}_x)}
\|f\|_{L^{\infty}_T L^{2}_\xi L^\infty_x},
\end{split}
\end{equation*}
where $c_1(j)$ is defined in \eqref{def.c1j},
and {the same type of convolution inequality for series as in \eqref{con.ine}} has been used in the last inequality.
Similarly, one can see that $I_{3,2}$ is also bounded as
\begin{equation*}
\begin{split}
I_{3,2}\lesssim\|f\|_{\widetilde{L}^{2}_T \widetilde{L}^{2}_{\xi,\nu} (B^{s}_x)}
\|g\|_{L^{\infty}_T L^{2}_\xi L^\infty_x}.
\end{split}
\end{equation*}

Combing all above estimates on $I_1,I_2$ and $I_3$, we obtain the inequality \eqref{basic.non.p1.}.
Hence, the proof of Lemma \ref{basic.non.es.} is completed.\end{proof}


Having Lemma \ref{basic.non.es.}, one can see that the following result also holds true.

\begin{lemma}\label{nh-h.non.es.}
Assume $s> 0$, $0\leq T\leq +\infty$, and let $f=f(t,x,\xi), g=g(t,x,\xi)$, and $h=h(t,x,\xi)$ be some suitable smooth distribution functions such that the following norms are well defined, then it holds that 
\begin{multline}
\label{nop.nhb.es.}
\sum\limits_{q\geq-1}2^{qs}\left[\int_0^T\left|(\Delta_q\Gamma(f,g),\Delta_q h)\right|dt\right]^{1/2}\lesssim \|h\|^{1/2}_{\widetilde{L}^{2}_T \widetilde{L}^{2}_{\xi,\nu} (B^{s}_x)}\\
\qquad\times\Bigg[
\dis \|g\|^{1/2}_{\widetilde{L}^{2}_T\widetilde{L}^{2}_{\xi,\nu}(B^{s}_x)}
\|f\|^{1/2}_{\widetilde{L}^{\infty}_T \widetilde{L}^{2}_\xi(\FX)}
+\|f\|^{1/2}_{\widetilde{L}^{2}_T\widetilde{L}^{2}_{\xi,\nu}(\FX)}
\|g\|^{1/2}_{\widetilde{L}^{\infty}_T \widetilde{L}^{2}_\xi (B^{s}_x)}\\[3mm]
+ \|f\|^{1/2}_{\widetilde{L}^{2}_T \widetilde{L}^{2}_{\xi,\nu} (B^{s}_x)}
\|g\|^{1/2}_{\widetilde{L}^{\infty}_T \widetilde{L}^{2}_\xi(\FX)}+\|g\|^{1/2}_{\widetilde{L}^{2}_T \widetilde{L}^{2}_{\xi,\nu}(\FX)}
\|f\|^{1/2}_{\widetilde{L}^{\infty}_T \widetilde{L}^{2}_\xi (B^{s}_x)}
\Bigg],
\end{multline}
where $\FX$ denotes either the inhomogeneous critical Besov space $B^{3/2}_x$ or the homogeneous critical Besov space  $\dot{B}^{3/2}_x$.
\end{lemma}

\begin{proof}
Noticing that $B^{3/2}_x\subset L^\infty_x$ and $\dot{B}^{3/2}_x\subset L^\infty_x$,
\eqref{nop.nhb.es.} follows from
\eqref{basic.non.p1.} in Lemma \ref{basic.non.es.} and \eqref{BCLb} in Lemma \ref{BCLsb}. This ends the proof of Lemma \ref{nh-h.non.es.}.
\end{proof}

The following is an immediate corollary of  Lemma \ref{basic.non.es.} and Lemma \ref{nh-h.non.es.}.

\begin{corollary}\label{nh-h.non.es2.}
Assume $s>0$, $0<T\leq \infty$. Let $f=f(t,x,\xi)$,  $g=g(t,x,\xi)$, and $h=h(t,x,\xi)$ be three suitably smooth distribution functions such that all the  norms on the right  of the following inequalities are well defined, then it holds that
\begin{multline}
\label{mi-ma-non.es1.}
\sum\limits_{q\geq-1}2^{qs}\left[\int_0^T\left|(\Delta_q\Gamma(\FP f, g),\Delta_q h)\right|dt\right]^{1/2}
\lesssim \|h\|^{1/2}_{\widetilde{L}^{2}_T \widetilde{L}^{2}_{\xi,\nu} (B^{s}_x)}\\
\times \left[\|g\|^{1/2}_{\widetilde{L}^{2}_T\widetilde{L}^{2}_{\xi,\nu} (B^{s}_x)}
\|\FP f\|^{1/2}_{\widetilde{L}^{\infty}_T\widetilde{L}^{2}_\xi(\FX)}+\|g\|^{1/2}_{\widetilde{L}^{2}_T\widetilde{L}^{2}_{\xi,\nu}(\FX) }
\|\FP f\|^{1/2}_{\widetilde{L}^{\infty}_T\widetilde{L}^{2}_\xi (B^{s}_x)}\right],
\end{multline}
\begin{multline}
\label{mi-ma-non.es2.}
\sum\limits_{q\geq-1}2^{qs}\left[\int_0^T\left|(\Delta_q\Gamma(f,\FP g),\Delta_q h)\right|dt\right]^{1/2}
\lesssim \|h\|^{1/2}_{\widetilde{L}^{2}_T \widetilde{L}^{2}_{\xi,\nu} (B^{s}_x)}\\
\times \left[\|f\|^{1/2}_{\widetilde{L}^{2}_T \widetilde{L}^{2}_{\xi,\nu}(B^{s}_x)}
\|\FP g\|^{1/2}_{\widetilde{L}^{\infty}_T \widetilde{L}^{2}_\xi(\FX)}+\|f\|^{1/2}_{\widetilde{L}^{2}_T \widetilde{L}^{2}_{\xi,\nu}(\FX)}
\|\FP g\|^{1/2}_{\widetilde{L}^{\infty}_T \widetilde{L}^{2}_\xi (B^{s}_x)}\right],
\end{multline}
and
\begin{multline}\label{mi-ma-non.es3.}
\sum\limits_{q\geq-1}2^{qs}\left[\int_0^T\left|(\Delta_q\Gamma(\FP f,\FP g),\Delta_q h)\right|dt\right]^{1/2}
\lesssim \|h\|^{1/2}_{\widetilde{L}^{2}_T \widetilde{L}^{2}_{\xi,\nu} (B^{s}_x)}\\
\times \left[\|\FP g\|^{1/2}_{\widetilde{L}^{\infty}_T \widetilde{L}^{2}_{\xi}(B^{s}_x)}
\|\FP f\|^{1/2}_{\widetilde{L}^{2}_T \widetilde{L}^{2}_\xi (\FX)}+\|\FP g\|^{1/2}_{\widetilde{L}^{\infty}_T \widetilde{L}^{2}_{\xi}(\FX)}
\|\FP f\|^{1/2}_{\widetilde{L}^{2}_T \widetilde{L}^{2}_\xi (B^{s}_x)}\right],
\end{multline}
where $\FX$ denotes either $B^{3/2}_x$ or $\dot{B}^{3/2}_x$ as in Lemma \ref{nh-h.non.es.}.
\end{corollary}

\begin{proof}
We prove the first estimate  \eqref{mi-ma-non.es1.} only, since \eqref{mi-ma-non.es2.} and \eqref{mi-ma-non.es3.} can be obtained in the same way. In fact,
\eqref{mi-ma-non.es1.} follows directly from Lemma \ref{basic.non.es.} with a slightly modification.
Applying $I$ from Lemma \ref{basic.non.es.} with $f=\FP f$,
and noticing that
$$
\|\Delta_q\FP f\|_{L^2_{\xi,\nu}L^2_x}\thicksim\|\Delta_q\FP f\|_{L^2_{\xi}L^2_x},\ \  \|S_j\FP f\|_{L^2_{\xi,\nu}L^2_x}\thicksim\|S_j\FP f\|_{L^2_{\xi}L^2_x},
$$
one can always take the $L^\infty_t-$norm of the terms involving $\FP f$, so that it is not necessary to exchange the
$L^\infty_t-$norm or $L^2_t-$norm of $\FP f$ or $g$. By this means, \eqref{mi-ma-non.es1.} can be verified through a tedious calculation,
we omit the details for brevity. This completes the proof of Corollary \ref{nh-h.non.es2.}.
\end{proof}

\section{Estimate on nonlinear term}

In this section we give the estimates on the nonlinear term $\Ga(f,f)$ and an estimate on the upper bound of $L f$. Recall \eqref{def.et} and \eqref{def.dt}.


\begin{lemma}\label{non.ED.} It holds that
\begin{eqnarray}\label{non.ED1.}
\sum\limits_{q\geq-1}2^{\frac{3q}{2}}\left[\int_0^T\left|(\Delta_q\Gamma(f,f),\Delta_q \{\FI-\FP\}f)\right|dt\right]^{1/2}
\lesssim\sqrt{\mathcal {E}_T (f)}\mathcal {D}_T(f),
\end{eqnarray}
for any $T>0$.
\end{lemma}

\begin{proof}
By the splitting $f=\FP f+\{\FI-\FP\}f$, we have
\begin{equation}\label{nonopsplit}
\begin{split}
\Gamma(f,f)
=&\Gamma({\bf P}f,{\bf P}f)+\Gamma({\bf P}f,\{{\bf I}-{\bf P}\}f)+\Gamma(\{{\bf I}-{\bf P}\}f,{\bf P}f)\\
&
+\Gamma(\{{\bf I}-{\bf P}\}f,\{{\bf I}-{\bf P}\}f).
\end{split}
\end{equation}
It now suffices to compute the left hand of \eqref{non.ED1.} in terms of the corresponding four terms on the right of \eqref{nonopsplit}.
In light of Corollary \ref{nh-h.non.es2.}, one can see that
\begin{multline*}
\sum\limits_{q\geq-1}2^{\frac{3q}{2}}\left[\int_0^T\left|(\Delta_q\Gamma(\FP f,\FP f),\Delta_q \{{\bf I}-{\bf P}\}f)\right|dt\right]^{1/2}\\
\lesssim\|\FP f\|^{1/2}_{\widetilde{L}^{2}_T \widetilde{L}^{2}_{\xi} (\dot{B}^{3/2}_x)}
\|\FP f\|^{1/2}_{\widetilde{L}^{\infty}_T \widetilde{L}^{2}_\xi (B^{3/2}_x)}
\|\{{\bf I}-{\bf P}\}f\|^{1/2}_{\widetilde{L}^{2}_T \widetilde{L}^{2}_{\xi,\nu}(B^{3/2}_x)}
\lesssim\sqrt{\mathcal {E}_T(f)}\mathcal {D}_T(f),
\end{multline*}
where we have used Lemma \ref{n-h-b} to ensure
\begin{multline}\notag
\|\FP f\|^{1/2}_{\widetilde{L}^{2}_T\widetilde{L}^{2}_{\xi} (\dot{B}^{3/2}_x)}\lesssim \|(a, b,c)\|^{1/2}_{\widetilde{L}^{2}_T (\dot{B}^{3/2}_x)}\sim \|\na_x(a, b,c)\|^{1/2}_{\widetilde{L}^{2}_T {(\dot{B}^{1/2}_x)}}\\
\lesssim\|\na_x(a, b,c)\|^{1/2}_{\widetilde{L}^{2}_T (B^{1/2}_x)}\lesssim\sqrt{\mathcal {D}_T(f)}.
\end{multline}
In a similar way, we next get from Corollary \ref{nh-h.non.es2.} that
\begin{multline*}
\sum\limits_{q\geq-1}2^{\frac{3q}{2}}\left[\int_0^T\left|(\Delta_q\Gamma(\FP f,\{{\bf I}-{\bf P}\}f),\Delta_q \{{\bf I}-{\bf P}\}f)\right|dt\right]^{1/2}
\\[2mm]
\lesssim \|\{{\bf I}-{\bf P}\}f\|^{1/2}_{\widetilde{L}^{2}_T \widetilde{L}^{2}_{\xi,\nu} (B^{3/2}_x)}
\|\FP f\|^{1/2}_{\widetilde{L}^{\infty}_T \widetilde{L}^{2}_{\xi} (B^{3/2}_x)}
\|\{{\bf I}-{\bf P}\}f\|^{1/2}_{\widetilde{L}^{2}_T \widetilde{L}^{2}_{\xi,\nu}(B^{3/2}_x)},
\end{multline*}
\begin{multline*}
\sum\limits_{q\geq-1}2^{\frac{3q}{2}}\left[\int_0^T\left|(\Delta_q\Gamma(\{{\bf I}-{\bf P}\} f,\FP f),\Delta_q \{{\bf I}-{\bf P}\}f)\right|dt\right]^{1/2}
\\[2mm]
\lesssim\|\FP f\|^{1/2}_{\widetilde{L}^{\infty}_T \widetilde{L}^{2}_{\xi}(B^{3/2}_x)}
\|\{{\bf I}-{\bf P}\}f\|^{1/2}_{\widetilde{L}^{2}_T \widetilde{L}^{2}_{\xi,\nu}(B^{3/2}_x)}
\|\{{\bf I}-{\bf P}\}f\|^{1/2}_{\widetilde{L}^{2}_T \widetilde{L}^{2}_{\xi,\nu}(B^{3/2}_x)},
\end{multline*}
and
\begin{multline*}
\sum\limits_{q\geq-1}2^{\frac{3q}{2}}\left[\int_0^T\left|(\Delta_q\Gamma(\{{\bf I}-{\bf P}\} f,\{{\bf I}-{\bf P}\}f),\Delta_q \{{\bf I}-{\bf P}\}f)\right|dt\right]^{1/2}
\\
\lesssim{\|\{{\bf I}-{\bf P}\}f\|^{1/2}_{\widetilde{L}^{\infty}_T \widetilde{L}^{2}_{\xi}(B^{3/2}_x)}
\|\{{\bf I}-{\bf P}\} f\|_{\widetilde{L}^{2}_T \widetilde{L}^{2}_{\xi,\nu}(B^{3/2}_x)}}.
\end{multline*}
Furthermore, it is straightforward to see that the above three estimates can be further bounded by $\sqrt{\mathcal {E}_T (f)}\mathcal {D}_T (f)$ up to a generic constant. Therefore \eqref{non.ED1.} follows from all the above estimates, and this completes the proof of Lemma  \ref{non.ED.}.
\end{proof}

The following lemma will be used in the process of  deducing the macroscopic dissipation rates in the next section.

\begin{lemma}\label{non.EDL2.}
Let $\zeta=\zeta(\xi)\in \mathcal {S}(\R^3_x)$ and $0<s\leq 3/2$. Then
it holds that
\begin{equation}\label{non.EDL21.}
\sum\limits_{q\geq-1}2^{qs}\left[\int_0^T\left\|\Delta_q\left( \Gamma(f,f),\zeta\right) \right\|_{L^2_x}^2dt\right]^{1/2}
\lesssim\mathcal {E}_T (f)\mathcal {D}_T (f),
\end{equation}
for any $T>0$, where the inner product $(\cdot,\cdot)$ on the left is taken with respect to velocity variable $\xi$ only.
\end{lemma}

\begin{proof}
We first consider the general case of $\Gamma(f,g)$ instead of $\Ga(f,f)$.
By H\"older's inequality and the change of variable $(\xi,\xi_\ast)\to(\xi',\xi'_\ast)$, it follows that
\begin{equation}\label{non.L2.}
\begin{split}
\left|\Delta_q\left(\Gamma(f,g),\zeta\right) \right|
\leq& \left(\int_{\R^6}d\xi d\xi_{\ast}\,
|\xi-\xi_{\ast}|^\ga\mu^{1/2}(\xi'_\ast)\left|\Delta_q[f_{\ast}g]\right|^2\right)^{1/2}\\
&\qquad \qquad \times\left(\int_{\R^6}d\xi d\xi_{\ast}\,
|\xi-\xi_{\ast}|^\ga\mu^{1/2}(\xi_\ast)|\zeta(\xi)|^2\right)^{1/2}
\\&+\left(\int_{\R^6}d\xi d\xi_{\ast}\,
|\xi-\xi_{\ast}|^\ga\mu^{1/2}(\xi_\ast)\left|\Delta_q[f_{\ast}g]\right|^2\right)^{1/2}\\
&\qquad\qquad\quad \times\left(\int_{\R^6}d\xi d\xi_{\ast}\,
|\xi-\xi_{\ast}|^\ga\mu^{1/2}(\xi_\ast)|\zeta(\xi)|^2\right)^{1/2}\\
\lesssim& \left(\int_{\R^6}d\xi d\xi_{\ast}\,
|\xi-\xi_{\ast}|^\ga\left|\Delta_q[f_{\ast}g]\right|^2\right)^{1/2}.
\end{split}
\end{equation}
With \eqref{non.L2.} in hand, one can further deduce
\begin{equation}\notag
\begin{split}
\sum\limits_{q\geq-1}2^{qs}&\left[\int_0^T\left\|\Delta_q\left(\Gamma(f,g),\zeta\right)\right\|_{L^2_x}^2dt\right]^{1/2}\\
\lesssim& \sum\limits_{q\geq-1}2^{qs}\left(\int_0^Tdt\int_{\R^9}dxd\xi d\xi_{\ast}\,
|\xi-\xi_{\ast}|\left|\Delta_q[f_{\ast}g]\right|^2\right)^{1/2}:=I.
\end{split}
\end{equation}
Recalling that we have obtained the estimates for $I$ in the
proof of Lemma  \ref{basic.non.es.} and Lemma \ref{nh-h.non.es.}, i.e.
\begin{equation}\label{non.L2.p2}{
\begin{split}
I\lesssim&\bigg[\|g\|_{\widetilde{L}^{2}_T\widetilde{L}^{2}_{\xi,\nu} (B^{s}_x)}
\|f\|_{\widetilde{L}^{\infty}_T \widetilde{L}^{2}_\xi(\FX)}
+\|f\|_{\widetilde{L}^{2}_T \widetilde{L}^{2}_{\xi,\nu}(\FX)}
\|g\|_{\widetilde{L}^{\infty}_T \widetilde{L}^{2}_\xi (B^{s}_x)}\\[3mm]
&\qquad\qquad + \|g\|_{\widetilde{L}^{2}_T \widetilde{L}^{2}_{\xi,\nu}(\FX)}
\|f\|_{\widetilde{L}^{\infty}_T \widetilde{L}^{2}_\xi (B^{s}_x)}+\|f\|_{\widetilde{L}^{2}_T \widetilde{L}^{2}_{\xi,\nu} (B^{s}_x)}
\|g\|_{\widetilde{L}^{\infty}_T \widetilde{L}^{2}_\xi(\FX)}
\bigg],
\end{split}}
\end{equation}
where $\FX$ denotes either $B^{3/2}_x$ or $\dot{B}^{3/2}_x$.

In particular, if $\Ga(\FP f,\FP f)$ is considered, it follows from Corollary \ref{nh-h.non.es2.} that
\begin{equation}\notag
\begin{split}
I\lesssim\|\FP f\|_{\widetilde{L}^{\infty}_T \widetilde{L}^{2}_{\xi} (B^{s}_x)}
\|\FP f\|_{\widetilde{L}^{2}_T \widetilde{L}^{2}_\xi (\dot{B}^{3/2}_x)}\lesssim\mathcal {E}_T (f)\mathcal {D}_T (f),
\end{split}
\end{equation}
where we have used $0<s\leq 3/2$.
Recalling the splitting \eqref{nonopsplit} and applying \eqref{non.L2.p2} and Corollary \ref{nh-h.non.es2.}, the other three terms corresponding to the splitting \eqref{nonopsplit}
can be computed as:
\begin{multline}\notag
\sum\limits_{q\geq-1}2^{qs}\left[\int_0^T\left\|\Delta_q\left (\Gamma(\FP f,\{{\bf I}-{\bf P}\} f),\zeta\right) \right\|_{L^2_x}^2dt\right]^{1/2}
\\[2mm]+\sum\limits_{q\geq-1}2^{qs}\left[\int_0^T\left\|\Delta_q\left (\Gamma(\{{\bf I}-{\bf P}\} f,\FP f),\zeta\right) \right\|_{L^2_x}^2dt\right]^{1/2}\\[2mm]
\lesssim\|\{{\bf I}-{\bf P}\}f\|_{\widetilde{L}^{2}_t\widetilde{L}^{2}_{\xi,\nu}(B^{s}_x)}
\|\FP f\|_{\widetilde{L}^{\infty}_T \widetilde{L}^{2}_{\xi} (B^{3/2}_x)}+\|\{{\bf I}-{\bf P}\}f\|_{\widetilde{L}^{2}_t\widetilde{L}^{2}_{\xi,\nu}(B^{3/2}_x)}
\|\FP f\|_{\widetilde{L}^{\infty}_T \widetilde{L}^{2}_{\xi} (B^{s}_x)},
\end{multline}
and
\begin{multline}\notag
\sum\limits_{q\geq-1}2^{qs}\left[\int_0^T\left\|\Delta_q\left (\Gamma(\{{\bf I}-{\bf P}\} f,\{{\bf I}-{\bf P}\} f),\zeta\right) \right\|_{L^2_x}^2dt\right]^{1/2}
\\
\lesssim\|\{{\bf I}-{\bf P}\}f\|_{\widetilde{L}^{2}_T \widetilde{L}^{2}_{\xi}(B^{s}_x)}
\|\{{\bf I}-{\bf P}\} f\|_{\widetilde{L}^{\infty}_T \widetilde{L}^{2}_{\xi,\nu}(B^{3/2}_x)}\\+
\|\{{\bf I}-{\bf P}\}f\|_{\widetilde{L}^{2}_T \widetilde{L}^{2}_{\xi}(B^{3/2}_x)}
\|\{{\bf I}-{\bf P}\} f\|_{\widetilde{L}^{\infty}_T \widetilde{L}^{2}_{\xi,\nu}(B^{s}_x)},
\end{multline}
which both can be further bounded by $\mathcal {E}_T (f)\mathcal {D}_T (f)$ up to a generic constant.
Combing all the above estimates, we obtain \eqref{non.EDL21.}.
This completes the proof of Lemma \ref{non.EDL2.}.
\end{proof}

Finally we give an estimate on the upper bound of the  linear term $Lf$.

\begin{lemma}\label{upd.L}
Let $\zeta=\zeta(\xi)\in \mathcal {S}(\R^3_\xi)$ and {$s>0$}. Then
it holds that
\begin{equation}\label{upd.L.p1}
\sum\limits_{q\geq-1}2^{qs}\left[\int_0^T\left\|\Delta_q\left ( L\{\FI-\FP\}f,\zeta\right) \right\|_{L^2_x}^2dt\right]^{1/2}
\lesssim\|\{{\bf I}-{\bf P}\} f\|_{\widetilde{L}^{2}_T \widetilde{L}^{2}_{\xi,\nu}(B^{s}_x)}.
\end{equation}
for any $T>0$,  where the inner product $(\cdot,\cdot)$ on the left is taken with respect to velocity variable $\xi$ only.
\end{lemma}
\begin{proof}
Since $L\{\FI-\FP\}f$ can be rewritten as
$$
L\{\FI-\FP\}f=-\left\{\Gamma(\{\FI-\FP\}f,\mu^{1/2})+\Gamma(\mu^{1/2},\{\FI-\FP\}f)\right\},
$$
then \eqref{upd.L.p1} follows directly from similar estimates as in the proof of Lemma \ref{non.EDL2.}. 
This ends the proof of Lemma \ref{upd.L}.
\end{proof}

\section{Estimate on macroscopic dissipation}

In this section, we would obtain the  macroscopic dissipation rate basing on the equation \eqref{f.eq.}.

\begin{lemma}\label{macro.dis.}
It holds that
\begin{multline}\label{macro.dis.p1}
\|\na_x(a,b,c)\|_{\widetilde{L}^{2}_T{(B^{1/2}_x)}}
\lesssim \|f_0\|_{\widetilde{L}^{2}_{\xi}(B^{3/2}_x)}+\CE_T(f)
\\ +\|\{\FI-\FP\}f\|_{\widetilde{L}^{2}_T \widetilde{L}^{2}_{\xi,\nu} (B^{3/2}_x)}+\mathcal {E}_T (f)\mathcal {D}_T (f),
\end{multline}
for any $T>0$.
\end{lemma}

\begin{proof}
First,
as in \cite{D-11,DS-VPB},
by taking the following velocity moments
\begin{equation*}
    \mu^{1/2}, \xi_i\mu^{1/2}, \frac{1}{6}(|\xi|^2-3)\mu^{1/2},
    (\xi_i{\xi_m}-1)\mu^{1/2}, \frac{1}{10}(|\xi|^2-5)\xi_i \mu^{1/2}
\end{equation*}
with {$1\leq i,m\leq 3$} for the equation \eqref{f.eq.},
the coefficient functions $(a,b,c)$ of the macroscopic component $\FP f$ given by \eqref{Pf} satisfy
the fluid-type system
\begin{equation}\label{mac.law}
    \left\{\begin{array}{l}
      \dis \pa_t a +\na_x\cdot b =0,\\[3mm]
      \dis \pa_t b +\na_x (a+2c)+\na_x\cdot \highG (\{\FI-\FP\} f)=0,\\[3mm]
      \dis \pa_t c +\frac{1}{3}\na_x\cdot b +\frac{1}{6}\na_x\cdot
      \Lambda (\{\FI-\FP\} f)=0,\\[3mm]
      \dis \pa_t[\highG_{{ im}}(\{\FI-\FP\} f)+2c\de_{{ im}}]+\pa_ib_m+\pa_m
      b_i=\highG_{im}(\mathbbm{r}+\mathbbm{h}),\\[3mm]
      \dis \pa_t \Lambda_i(\{\FI-\FP\} f)+\pa_i c = \Lambda_i(\mathbbm{r}+\mathbbm{h}),
    \end{array}\right.
\end{equation}
where the
high-order moment functions $\highG=(\highG_{im}(\cdot))_{3\times 3}$ and
$\highB=(\highB_i(\cdot))_{1\leq i\leq 3}$ are defined by
\begin{eqnarray}
  \highG_{im}(f) = \left ((\xi_i\xi_m-1)\mu^{1/2}, f\right),\ \ \
  \highB_i(f)=\frac{1}{10}\left ((|\xi|^2-5)\xi_i\mu^{1/2},
  f\right),\notag
\end{eqnarray}
with the inner produce taken with respect to velocity variable $\xi$ only, and the terms $\mathbbm{r}$ and $\mathbbm{h}$ are given by
\begin{eqnarray}\notag
\mathbbm{r}= -\xi\cdot \na_x \{\FI-\FP\}f,\ \ \mathbbm{h}=-L \{\FI-\FP\}f+\Gamma(f,f).
\end{eqnarray}
Applying $\Delta_q$ with $q\geq-1$ to the system \eqref{mac.law}, we obtain
\begin{equation}\notag
    \left\{\begin{array}{l}
      \dis \pa_t \Delta_qa +\na_x\cdot \Delta_qb =0,\\[3mm]
      \dis \pa_t \Delta_qb +\na_x \Delta_q(a+2c)+\na_x\cdot \highG (\Delta_q\{\FI-\FP\} f)=0,\\[3mm]
      \dis \pa_t \Delta_qc +\frac{1}{3}\na_x\cdot \Delta_qb +\frac{1}{6}\na_x\cdot
      \Lambda (\Delta_q\{\FI-\FP\} f)=0,\\[3mm]
      \dis \pa_t [\highG_{im}(\Delta_q\{\FI-\FP\} f)+2\Delta_qc\de_{im}]+\pa_i\Delta_qb_m+\pa_m
      \Delta_qb_i=\highG_{im}(\Delta_q\mathbbm{r}+\Delta_q\mathbbm{h}),\\[3mm]
      \dis \pa_t \Lambda_i(\Delta_q\{\FI-\FP\} f)+\pa_i \Delta_qc = \Lambda_i(\Delta_q\mathbbm{r}+\Delta_q\mathbbm{h}).
    \end{array}\right.
\end{equation}
Now, in a similar way as in \cite{D-08}, one can prove from the above system that
\begin{multline}
 \frac{d}{dt}\mathcal {E}^{int}_{q}(f(t))+\la\|\Delta_q \na_x (a,b,c)\|_{L^2_x}^2
\lesssim\|\Delta_q\na_x\{\FI-\FP\}  f\|_{L^2_\xi L^2_x}^2\\
+C\sum_{i=1}^3\left\|\Lambda_i\left(\Delta_q\mathbbm{h}\right)\right\|_{L^2_x}^2
+C\sum_{i,m=1}^3\|\highG_{im}(\Delta_q\mathbbm{h})\|_{L^2_x}^2,\label{lem.m3f.sum1}
\end{multline}
for any $t\geq 0$, where the temporal interactive functional $\mathcal {E}^{int}_{q}(f(t))$ for each $q\geq -1$ is defined by
\begin{eqnarray}
\mathcal {E}^{int}_{q}(f(t))&=&\sum_{i=1}^3 (\Delta_q \pa_ic,  \Lambda_i(\Delta_q\{\FI-\FP\}  f))\notag\\
&&+\kappa_1\sum_{i,m=1}^3 ( \pa_i\Delta_q b_m+\pa_m\Delta_q b_i , \highG_{im}(\{\FI-\FP\}\Delta_q f)\notag\\
&&+\kappa_2\sum_{i=1}^3 (\Delta_q \pa_ia, \Delta_q b_i),
\label{inenq.def.}
\end{eqnarray}
with suitably chosen constants $0<\kappa_2\ll\kappa_1\ll1$.
Integrating \eqref{lem.m3f.sum1} with respect to $t$ over $[0,T]$ and taking the square roots of both sides of the resulting inequality, one has
\begin{multline}
\left(\int_0^T\|\Delta_q \na_x (a,b,c)\|_{L^2_x}^2dt\right)^{1/2}\\
\lesssim\sqrt{|\mathcal {E}^{int}_{q}(f(T))|}+C\sqrt{|\mathcal {E}^{int}_{q}(f(0))|}
+\|\Delta_q\na_x\{\FI-\FP\}  f\|_{L^2_TL^2_\xi L^2_x}\\
+\sum_{i=1}^3\left\|\Lambda_i\left(\Delta_q\mathbbm{h}\right)\right\|_{L_T^2L^2_x}
+\sum_{i,m=1}^3\|\highG_{im}(\Delta_q\mathbbm{h})\|_{L_T^2L^2_x}.\label{lem.m3f.sum2}
\end{multline}
Now multiplying \eqref{lem.m3f.sum2} by $2^{q/2}$ and taking the summation over $q\geq-1$ gives
\begin{multline}
\|\na_x (a,b,c)\|_{\widetilde{L}^2_T (B^{1/2}_x)}
\lesssim\sum\limits_{q\geq-1}2^{q/2}\sqrt{|\mathcal {E}^{int}_{q}(f(T))|}+\sum\limits_{q\geq-1}2^{q/2}\sqrt{|\mathcal {E}^{int}_{q}(f(0))|}\\+\|\{\FI-\FP\}f\|_{\widetilde{L}^{2}_T\widetilde{L}^{2}_{\xi,\nu} (B^{3/2}_x)}
+\sum_{i=1}^3\sum\limits_{q\geq-1}2^{q/2}\left\|\Lambda_i\left(\Delta_q\mathbbm{h}\right)\right\|_{L_T^2L^2_x}
\\+\sum_{i,m=1}^3\sum\limits_{q\geq-1}2^{q/2}\|\highG_{im}(\Delta_q\mathbbm{h})\|_{L_T^2L^2_x}.\label{lem.m3f.sum3}
\end{multline}
Furthermore, it follows from \eqref{inenq.def.}, Lemma \ref{B.IM.} and {the} Cauchy-Schwarz inequality that
\begin{multline}\notag%
\sum\limits_{q\geq-1}2^{q/2}\sqrt{|\mathcal {E}^{int}_{q}(f(t))|}
\lesssim
\sum\limits_{q\geq-1}2^{q/2}\{\|\Delta_q\na_x (a,b,c)(t)\|_{L_x^2}+\|\Delta_qb (t)\|_{L_x^2}\\
+\|\Delta_q\na_x\{\FI-\FP\}  f(t)\|_{L^2_\xi L^2_x}\},
\end{multline}
for any $0\leq t\leq T$, which implies
\begin{equation}
\label{macro.eg.}
\sum\limits_{q\geq-1}2^{q/2}\sqrt{|\mathcal {E}^{int}_{q}(f(T))|}\lesssim \mathcal {E}_T(f),\quad \sum\limits_{q\geq-1}2^{q/2}\sqrt{|\mathcal {E}^{int}_{q}(f(0))|}\lesssim\|f_0\|_{\widetilde{L}^{2}_{\xi}(B^{3/2}_x)}.
\end{equation}
And, Lemma \ref{non.EDL2.} and Lemma \ref{upd.L} imply
\begin{multline}\label{marco.non.}
\sum\limits_{q\geq-1}2^{q/2}\left\|\Lambda_i\left(\Delta_q\mathbbm{h}\right)\right\|_{L_T^2L^2_x}
+\sum\limits_{q\geq-1}2^{q/2}\|\highG_{im}(\Delta_q\mathbbm{h})\|_{L_T^2L^2_x}
\\
\lesssim\|\{\FI-\FP\}f\|_{\widetilde{L}^{2}_T \widetilde{L}^{2}_{\xi,\nu} (B^{3/2}_x)}+\mathcal {E}_T (f)\mathcal {D}_T (f).
\end{multline}
Therefore \eqref{macro.dis.p1} follows from \eqref{lem.m3f.sum3} with the help of \eqref{macro.eg.} and \eqref{marco.non.}. This completes the proof of Lemma \ref{macro.dis.}.
\end{proof}

\section{Global a priori estimate}
This section is devoted to deducing the global a priori estimate for the Boltzmann equation \eqref{f.eq.}.

\begin{lemma}\label{basic.eng.}
There is indeed an energy functional $\CE_T (f)$ satisfying \eqref{def.et} such that
\begin{equation}\label{basic.eng1.}
{\mathcal {E}_T }(f)+\mathcal {D}_T (f)\leq {C}\|f_0\|_{\widetilde{L}^{2}_{\xi}(B^{3/2}_x)}+
{C}\left\{\sqrt{\mathcal {E}_T (f)}+\mathcal {E}_T (f)\right\}\mathcal {D}_T (f),
\end{equation}
for any $T>0$, where ${C}$ is a constant independent of $T$.
\end{lemma}

\begin{proof}
Applying the operator $\Delta_q$ $(q\geq-1)$ to \eqref{f.eq.}, taking the inner product with $2^{3q}\Delta_q f$ over $\R^3_x\times \R^3_\xi$
and applying Lemma \ref{qL+}, we obtain
\begin{multline}\label{f.inner}
\frac{1}{2}\frac{d}{dt} 2^{3q}\|\Delta_qf\|^2_{L^2_\xi L^2_x}+{\la_{0}} 2^{3q}\|\Delta_q\{\FI-\FP\}f\|_{L^2_{\xi,\nu}L^2_{x}}^{2}\\
\leq2^{3q}|(\Delta_q\Gamma(f,f),\Delta_q\{\FI-\FP\}f)|.
\end{multline}
Integrating \eqref{f.inner}  over $[0,t]$ with $0\leq t\leq T$ and taking the square root of both sides of the resulting inequality
yield
\begin{multline}\notag
2^{\frac{3q}{2}}\|\Delta_qf(t)\|_{L^2_\xi L^2_x}+{\sqrt{\la_{0}}} 2^{\frac{3q}{2}}\left(\int_0^t\|\Delta_q\{\FI-\FP\}f\|_{L^2_{\xi,\nu}L^2_{x}}^{2}d\tau\right)^{1/2}
\\\leq 2^{\frac{3q}{2}}\|\Delta_qf_0\|_{L^2_\xi L^2_x}+2^{\frac{3q}{2}}\left(\int_0^t|(\Delta_q\Gamma(f,f),\Delta_q\{\FI-\FP\}f)|d\tau\right)^{1/2},
\end{multline}
for any $0\leq t\leq T$.
Further by taking the summation over $q\geq -1$, the above estimate implies
\begin{multline}\notag
\sum\limits_{q\geq-1}2^{\frac{3q}{2}}\sup_{0\leq t\leq T}\|\Delta_qf(t)\|_{L^2_\xi L^2_x}\\
+{\sqrt{\la_{0}}}\sum\limits_{q\geq-1}2^{\frac{3q}{2}}\left(\int_0^T \|\Delta_q\{\FI-\FP\}f\|_{L^2_{\xi,\nu}L^2_{x}}^{2}dt \right)^{1/2}
\\ \leq \sum\limits_{q\geq-1}2^{\frac{3q}{2}}\|\Delta_qf_0\|_{L^2_\xi L^2_x}+\sum\limits_{q\geq-1}2^{\frac{3q}{2}}\left(\int_0^T|(\Delta_q\Gamma(f,f),\Delta_q\{\FI-\FP\}f)|dt\right)^{1/2}.
\end{multline}
Due to Lemma \ref{non.ED.} it further follows that
\begin{equation}\label{f.inner4}
\|f\|_{\widetilde{L}^{\infty}_T \widetilde{L}^{2}_{\xi} (B^{3/2}_x)}
+{\sqrt{\la_{0}}}\|\{\FI-\FP\}f\|_{\widetilde{L}^{2}_T \widetilde{L}^{2}_{\xi,\nu}(B^{3/2}_x)}
\lesssim \|f_0\|_{\widetilde{L}^{2}_{\xi}(B^{3/2}_x)}
+\sqrt{\mathcal {E}_T (f)}\mathcal {D}_T (f),
\end{equation}
where $T>0$ can be arbitrary. Furthermore,  we recall Lemma \ref{macro.dis.}. By letting $0<\kappa_3\ll1$, we get from
$\eqref{macro.dis.p1}\times\kappa_3+\eqref{f.inner4}$ that
\begin{multline}\label{f.inner5}
\|f\|_{\widetilde{L}^{\infty}_T \widetilde{L}^{2}_{\xi}{(B^{3/2}_x)}}-\kappa_3\CE_T (f)
+\la \left\{\|\na_x(a,b,c)\|_{\widetilde{L}^{2}_T (B^{1/2}_x)}+\|\{\FI-\FP\}f\|_{\widetilde{L}^{2}_T \widetilde{L}^{2}_{\xi,\nu} (B^{3/2}_x)}\right\}
\\
\lesssim\|f_0\|_{\widetilde{L}^{2}_{\xi}(B^{3/2}_x)}
+\left\{\sqrt{\mathcal {E}_T (f)}+\mathcal {E}_T (f)\right\}\mathcal {D}_T (f).
\end{multline}
Therefore, \eqref{basic.eng1.} follows from \eqref{f.inner5}  by noticing
$$
\|f\|_{\widetilde{L}^{\infty}_T \widetilde{L}^{2}_{\xi}{(B^{3/2}_x)}}-\kappa_3\CE_T (f)\sim \CE_T (f),
$$
since $\ka_3>0$ can be small enough.
The proof of Lemma \ref{basic.eng.} is complete.
\end{proof}

\section{Local existence}

In this section, we will etablish the  local-in-time existence of solutions to the Boltzmann equation \eqref{f.eq.} in the space $\widetilde{L}^{\infty}_T\widetilde{L}^{2}_{\xi}(B^{3/2}_x)$ for $T>0$ small enough. The
construction of the local solution is based on a uniform energy estimate for the following
sequence of iterating approximate solutions:
\begin{eqnarray}\notag
\begin{split}
\left\{\begin{array}{rll}
\begin{split}
\left\{\partial_t+\xi\cdot\nabla_x \right\}F^{n+1}&+F^{n+1}(\xi)\int_{\R^3\times \S^2}|\xi-\xi_\ast|^\ga B_0(\theta)
F^n(\xi_\ast)\,d\xi_\ast d\omega\\&=\int_{\R^3\times \S^2}|\xi-\xi_\ast|^\ga B_0(\theta)
F^n
(\xi'_{\ast})F^n(\xi')\,d\xi_\ast d\omega,\\
F^{n+1}(0,x,\xi)&=F_0(x,\xi),
\end{split}
\end{array}\right.
\end{split}
\end{eqnarray}
starting with $F^0(t,x,\xi)=F_0(x,\xi)$.

Noticing that $F^{n+1}=\mu+\mu^{1/2}f^{n+1}$, equivalently we need to solve $f^{n+1}$ such that
\begin{equation}\label{iterate.f}
\begin{split}
\left\{\partial_t+\xi\cdot\nabla_x +\nu\right\}f^{n+1}-Kf^{n}&=\Gamma_{gain}(f^{n},f^{n})-\Gamma_{loss}(f^{n},f^{n+1}),\\
f^{n+1}(0,x,\xi)&=f_0(x,\xi).
\end{split}
\end{equation}
Our discussion is based on the uniform bound in $n$ for {$\mathcal
{E}_T (f^{n})$} for a small time $T>0$.
The crucial energy estimate is given as follows.

\begin{lemma}\label{fn.bdd.lem}
The solution sequence $\{f^{n}\}_{n=1}^\infty$ is well defined. For  a sufficiently small constant $M_0>0$,
there exists $T^{\ast}=T^{\ast}(M_0)>0$ such that if
$$
\|f_0\|_{\widetilde{L}^{2}_{\xi} (B^{3/2}_x)} \leq{M_0},
$$
then for any $n$, it holds that
\begin{equation}\label{fn.bdd}
\widetilde{Y}_T (f^n):=\mathcal {E}_T (f^{n})+\widetilde{\mathcal {D}}_T (f^{n})\leq 2M_0, \ \ \forall\,T\in[0,T^*),
\end{equation}
where $\widetilde{\mathcal
{D}}_T (f)$ is defined by
\begin{eqnarray*}
\widetilde{\mathcal
{D}}_T (f)=\|f\|_{\widetilde{L}^{2}_T \widetilde{L}^{2}_{\xi,\nu} (B^{3/2}_x)}.
\end{eqnarray*}
\end{lemma}

\begin{proof}
To prove \eqref{fn.bdd}, we use induction on $n$. Namely, for each integer $l\geq0$, we are going to verify:
\begin{equation}\label{yl.bdd}
\widetilde{Y}_T(f^l)\leq 2M_0,
\end{equation}
for $0\leq T<T^\ast$, where $M_0$ and $T^\ast>0$ are to be suitably chosen later on.
Clearly the case $l=0$ is valid. We assume (\ref{yl.bdd}) is true for $l=n$.
Applying $\Delta_q$ $(q\geq-1)$ to \eqref{iterate.f} and taking the inner product with $2^{3q}\Delta_qf^{n+1}$ over $\R^3_x\times\R^3_\xi$, one has
\begin{multline}\label{D.fn.inner}
\frac{d}{dt}2^{3q}\left\|\Delta_qf^{n+1}\right\|^2_{L^2_\xi L^2_x}+2^{3q+1}\left\|\Delta_qf^{n+1}\right\|_{L^2_{\xi,\nu}L^2_{x}}^{2}\\
=2^{3q+1} \Big(\Delta_q\Gamma_{gain}(f^{n},f^{n})- \Delta_q\Gamma_{loss}(f^{n},f^{n+1}) +K\Delta_qf^n, \Delta_qf^{n+1}\Big),
\end{multline}
which further implies
\begin{multline}\label{fn.inner}
\frac{d}{dt}2^{3q}\left\|\Delta_qf^{n+1}\right\|^2_{L^2_\xi L^2_x}+2^{3q+1}\left\|\Delta_qf^{n+1}\right\|_{L^2_{\xi,\nu}L^2_{x}}^{2}\\
\leq2^{3q+1}\left|\left(\Delta_q\Gamma_{gain}(f^{n},f^{n}),\Delta_qf^{n+1}\right)\right|\\
+2^{3q+1}\left|\left(\Delta_q\Gamma_{loss}(f^{n},f^{n+1}),\Delta_qf^{n+1}\right)\right|
+2^{3q+1}\left|\left(K\Delta_qf^n, \Delta_qf^{n+1}\right)\right|.
\end{multline}
Now by integrating \eqref{fn.inner} with respect to the time variable over $[0,t]$ with $0\leq t\leq T$, taking the square root of both sides of the resulting inequality
and summing up over $q\geq -1$,
one has
\begin{equation}\label{fn.eng1.}
\begin{split}
\sum\limits_{q\geq-1}&2^{\frac{3q}{2}}\sup_{0\leq t\leq T}\left\|\Delta_qf^{n+1}(t)\right\|_{L^2_\xi L^2_x}+\sum\limits_{q\geq-1}2^{\frac{3q+1}{2}}\left(\int_0^T\left\|\Delta_qf^{n+1}\right\|_{L^2_{\xi,\nu}L^2_{x}}^{2}dt\right)^{1/2}\\[2mm]
\leq&\sum\limits_{q\geq-1}2^{\frac{3q}{2}}\left\|\Delta_qf_0\right\|_{L^2_\xi L^2_x}
+\sum\limits_{q\geq-1}2^{\frac{3q+1}{2}}\left(\int_0^T\left|\left(\Delta_q\Gamma_{gain}(f^{n},f^{n}),\Delta_qf^{n+1}\right)\right|dt\right)^{1/2}
\\&+\sum\limits_{q\geq-1}2^{\frac{3q+1}{2}}\left(\int_0^T\left|\left(\Delta_q\Gamma_{loss}(f^{n},f^{n+1}),\Delta_qf^{n+1}\right)\right|dt\right)^{1/2}
\\&+\sum\limits_{q\geq-1}2^{\frac{3q+1}{2}}\left(\int_0^T\left|\left(K\Delta_qf^n, \Delta_qf^{n+1}\right)\right|dt\right)^{1/2}.
\end{split}
\end{equation}
From Lemma \ref{nh-h.non.es.} and Lemma \ref{K.es.},  \eqref{fn.eng1.}
implies
\begin{equation}\label{fn.eng2.}
\begin{split}
&\mathcal {E}_T (f^{n+1})+\widetilde{\mathcal {D}}_T (f^{n+1})\\
&\leq\|f_0\|_{\widetilde{L}^{2}_{\xi} (B^{3/2}_x)}
+C\sqrt{\mathcal {E}_T (f^{n})}\sqrt{\widetilde{\mathcal {D}}_T (f^{n})}\sqrt{\widetilde{\mathcal {D}}_T (f^{n+1})}
\\&\quad +C\sqrt{\mathcal {E}_T (f^{n})}\widetilde{\mathcal {D}}_T (f^{n+1})
+C\sqrt{\mathcal {E}_T (f^{n+1})}\sqrt{\widetilde{\mathcal {D}}_T (f^{n})}\sqrt{\widetilde{\mathcal {D}}_T (f^{n+1})}\\
&\quad +\sum\limits_{q\geq-1}2^{\frac{3q+1}{2}}
\left(\int_0^T\|\Delta_qf^n\|_{L^2_{\xi}L^2_x}\|\Delta_qf^{n+1}\|_{L^2_{\xi}L^2_x}dt\right)^{1/2}.
\end{split}
\end{equation}
The last term on the right hand side of \eqref{fn.eng2.} can be bounded by
\begin{equation}\label{fn.eng2.rh.5}
\begin{split}
&\sqrt{T}\sum\limits_{q\geq-1}2^{\frac{3q+1}{2}}
\left(\sup\limits_{0\leq t\leq T}\|\Delta_qf^n\|_{L^2_{\xi}L^2_x}\sup\limits_{0\leq t\leq T}\|\Delta_qf^{n+1}\|_{L^2_{\xi}L^2_x} \right)^{1/2}
\\ & \lesssim \sqrt{T}
\left(\sum\limits_{q\geq-1}2^{\frac{3q}{2}}\sup\limits_{0\leq t\leq T}\|\Delta_qf^n\|_{L^2_{\xi}L^2_x}\right)^{1/2}
\left(\sum\limits_{q\geq-1}2^{\frac{3q}{2}}\sup\limits_{0\leq t\leq T}\|\Delta_qf^{n+1}\|_{L^2_{\xi}L^2_x}\right)^{1/2}\\
& \lesssim \sqrt{T}\mathcal {E}_T (f^{n})+\sqrt{T}\mathcal {E}_T (f^{n+1}),
\end{split}
\end{equation}
where in the second line the discrete version of {the} Cauchy-Schwarz  inequality has been used. Using {the} Cauchy-Schwarz inequality,  the second and fourth terms on the right hand side of \eqref{fn.eng2.}  can be dominated by
\begin{equation}\label{fn.eng2.rh.2-4}
\begin{split}
\eta\widetilde{\mathcal {D}}_T (f^{n+1})
+\frac{C}{\eta}\mathcal {E}_T (f^{n})\widetilde{\mathcal {D}}_T (f^{n})
+C\sqrt{\widetilde{\mathcal {D}}_T (f^{n})}\left\{\mathcal {E}_T (f^{n+1})+\widetilde{\mathcal {D}}_T (f^{n+1})\right\},
\end{split}
\end{equation}
where $\eta$ is an arbitrary small positive constant.
By substituting \eqref{fn.eng2.rh.5} and \eqref{fn.eng2.rh.2-4} into \eqref{fn.eng2.}, applying the inductive hypothesis
and recalling $0\leq T<T^*$, it follows
\begin{multline}\notag
\left(1-C\sqrt{T^*}-C\sqrt{M_0}\right)\mathcal {E}_T (f^{n+1})+\left(1-\eta-2C\sqrt{M_0}\right)\widetilde{\mathcal {D}}_T (f^{n+1})\\
\leq M_0
+C\sqrt{T^*}M_0+\frac{C}{\eta}M_0^2.
\end{multline}
This then implies \eqref{yl.bdd} for $l=n+1$, since $\eta>0$ can be  small enough and both $T^*>0$ and $M_0>0$ are chosen to be  suitably small. The proof of Lemma \ref{fn.bdd.lem} is therefore complete.
\end{proof}

With the uniform bound on  the iterative solution sequence in terms of \eqref{iterate.f} by Lemma \ref{fn.bdd.lem}, we
can give the proof of  the local existence of solutions in the following theorem. We remark that the approach used here is due  to
Guo {\cite{Guo-L}}.

\begin{theorem}\label{local.existence}
Assume $0\leq\ga\leq1$. For a sufficiently small $M_0>0$, there exists $T^{\ast}=T^{\ast}(M_0)>0$ such that if
$$
\|f_0\|_{\widetilde{L}^{2}_{\xi} (B^{3/2}_x)} \leq{M_0},
$$
then there is a unique strong solution $f(t,x,\xi)$ to the Boltzmann equation (\ref{f.eq.})
in $(0,T^{\ast})\times {\R}^3_x\times {\R}^3_\xi$ with initial data $f(0,x,\xi)=f_0(x,\xi)$, such that
\begin{equation}\notag
\widetilde{Y}_T (f)\leq 2M_0,
\end{equation}
for any $T\in [0,T^\ast)$,
where $\widetilde{Y}_T (f)$ is defined in \eqref{fn.bdd}.
Moreover $\widetilde{Y}_T (f)$ is continuous in $T $ over $[0,T^*)$, and if $
F_0(x,\xi)=\mu+\mu^{1/2}f_{0}\geq0$, then
$F(t,x,\xi)=\mu+\mu^{1/2}f(t,x,\xi)\geq 0$ holds true.
\end{theorem}

\begin{proof}
In terms of \eqref{fn.bdd}, the limit function $f(t,x,\xi)$ of the approximate solution sequence $\{f^n\}_{n=1}^\infty$ must be the solution to \eqref{f.eq.} with $f(0,x,\xi)=f_0(x,\xi)$ in the sense of distribution. The distribution solution turns out to be a strong solution because
it  can be shown to be unique as follows.

To prove the uniqueness, we assume that another solution $g$ with the same initial data with $f$, i.e.~$g(0,x,\xi)=f_0(x,\xi)$, exists such that
$$
\widetilde{Y}_T (g)\leq 2M_0,
$$
on $T\in[0,T^*)$.  Taking the difference of the Boltzmann equation  \eqref{f.eq.} for $f$ and $g$, one has
\begin{equation*}
[\pa_t+\xi\cdot\na_x](f-g)+\nu(f-g)=\Gamma(f-g,f)+\Gamma(g,f-g)+K(f-g).
\end{equation*}
Then, by performing the completely same energy estimate as for obtaining \eqref{fn.eng2.}, it follows that
\begin{equation*}\label{fg.eng}
\begin{split}
&\widetilde{Y}_T (f-g)\\
&\lesssim 
{\sqrt{\mathcal {E}_T (f)+\mathcal {E}_T (g)}}\widetilde{\mathcal {D}}_T (f-g)\\
&\quad+\sqrt{\mathcal {E}_T (f-g)}{\sqrt{\widetilde{\mathcal {D}}_T (f)+\widetilde{\mathcal {D}}_T (g)}}\sqrt{\widetilde{\mathcal {D}}_T (f-g)}\\
&\quad+\sum\limits_{q\geq-1}2^{\frac{3q+1}{2}}
\left(\int_0^T\|\Delta_q(f-g)\|_{L^2_{\xi}L^2_x}\|\Delta_q(f-g)\|_{L^2_{\xi}L^2_x}dt\right)^{1/2}\\
&\lesssim
{\sqrt{\mathcal {E}_T (f)+\mathcal {E}_T (g)}}\widetilde{\mathcal {D}}_T (f-g)+\sqrt{\mathcal {E}_T (f-g)}{\sqrt{\widetilde{\mathcal {D}}_T (f)+\widetilde{\mathcal {D}}_T (g)}}\sqrt{\widetilde{\mathcal {D}}_T (f-g)}\\[2mm]
&\quad+\sqrt{T}\mathcal {E}_T (f-g)+\sqrt{T}\mathcal {E}_T (f-g).
\end{split}
\end{equation*}
We  therefore deduce $f\equiv g$
by letting $T< T^*$, because $\widetilde{Y}_T (f)\leq 2M_0$, $\widetilde{Y}_T (g)\leq 2M_0$, and $M_0$ and $T^\ast$ can be chosen suitably small.

To prove $T\mapsto \widetilde{Y}_T(f)$ is continuous in $[0,T^\ast)$, we first show that
\begin{equation}
\label{D.def.eft}
t\mapsto \CE(f(t)):=\sum_{q\geq -1} 2^{\frac{3q}{2}} \|\Delta_qf(t)\|_{L^2_{x,\xi}}
\end{equation}
is continuous on $[0,T^\ast)$. Indeed, take $t_1,t_2$ with $0\leq t_1,t_2<T^\ast$ and we let $t_1<t_2$ for brevity of presentation. By
letting $f^{n+1}=f^{n}=f$ in \eqref{D.fn.inner},
integrating the resulting inequality with respect to the time variable over $[t_1,t_2]$,
taking the square root of both sides
and then summing up over $q\geq -1$, similar to obtain \eqref{fn.eng2.}, one has
\begin{equation}
\notag%
\left|\CE(f(t_2))-\CE(f(t_1))\right|
\lesssim (\sqrt{M_0}+1) \sum\limits_{q\geq-1}2^{\frac{3q}{2}}\left(\int_{t_1}^{t_2}\|\Delta_qf\|^2_{L^2_{\xi,\nu}L^2_x}dt\right)^{1/2}.
\end{equation}
With this, it suffices to prove
\begin{equation}
\label{D.lim}
\lim\limits_{t_2\to t_1}\sum\limits_{q\geq-1}2^{\frac{3q}{2}}\left(\int_{t_1}^{t_2}\|\Delta_qf\|^2_{L^2_{\xi,\nu}L^2_x}dt\right)^{1/2}=0.
\end{equation}
Take $\eps>0$. Since $\sum_{q\geq-1}2^{\frac{3q}{2}}\left(\int_{0}^{T}\|\Delta_qf\|^2_{L^2_{\xi,\nu}L^2_x}dt\right)^{1/2}$ is finite for a fixed time $T$ with $\max\{t_1,t_2\}<T<T^\ast$, there is an integer $N$ such that
\begin{equation}
\notag%
\sum\limits_{q\geq N+1}2^{\frac{3q}{2}}\left(\int_{t_1}^{t_2}\|\Delta_qf\|^2_{L^2_{\xi,\nu}L^2_x}dt\right)^{1/2}\leq \sum\limits_{q\geq N+1}2^{\frac{3q}{2}}\left(\int_{0}^{T}\|\Delta_qf\|^2_{L^2_{\xi,\nu}L^2_x}dt\right)^{1/2}< \frac{\eps}{2}.
\end{equation}
On the other hand, it is straightforward to see
\begin{equation}
\notag%
\lim\limits_{t_2\to t_1}\sum\limits_{-1\leq q\leq N}2^{\frac{3q}{2}}\left(\int_{t_1}^{t_2}\|\Delta_qf\|^2_{L^2_{\xi,\nu}L^2_x}dt\right)^{1/2}=0,
\end{equation}
which implies that there is $\delta>0$ such that
\begin{equation}
\notag%
\sum\limits_{-1\leq q\leq N}2^{\frac{3q}{2}}\left(\int_{t_1}^{t_2}\|\Delta_qf\|^2_{L^2_{\xi,\nu}L^2_x}dt\right)^{1/2}< \frac{\eps}{2},
\end{equation}
whenever $|t_2-t_1|<\delta$. Hence, for $|t_2-t_1|<\delta$,
\begin{equation}\notag%
\sum\limits_{q\geq-1}2^{\frac{3q}{2}}\left(\int_{t_1}^{t_2}\|\Delta_qf\|^2_{L^2_{\xi,\nu}L^2_x}dt\right)^{1/2} = \left(\sum_{-1\leq q\leq N}+\sum_{q\geq N+1}\right)\cdots<\frac{\eps}{2}+\frac{\eps}{2}=\eps.
\end{equation}
Then \eqref{D.lim} is proved, and thus $t\mapsto \CE(f(t))$ is continuous on $[0,T^\ast)$. In particular, $\CE(f(t))$ given in \eqref{D.def.eft} is well defined for each $t\in [0,T^\ast)$. Notice that the same proof also yields that for each $q\geq -1$, the function $t\mapsto \|\Delta_qf(t)\|_{L^2_{x,\xi}}$ is continuous on $[0,T^\ast)$.

We now show that $T\mapsto \widetilde{Y}_T(f)$ is continuous in $[0,T^\ast)$. Indeed,  take $T_1,T_2$ with $0\leq T_1<T_2<T^\ast$.  Recall \eqref{fn.bdd}. Notice that $\widetilde{Y}_T(f)$ is nondecreasing in $T$. Then,
\begin{eqnarray*}
0&\leq &\widetilde{Y}_{T_2}(f)-\widetilde{Y}_{T_1}(f)\\
&=&\left(\|f\|_{\widetilde{L}^{\infty}_{T_2} \widetilde{L}^{2}_{\xi}{(B^{3/2}_x)}}-\|f\|_{\widetilde{L}^{\infty}_{T_1} \widetilde{L}^{2}_{\xi}{(B^{3/2}_x)}}\right)
+\left(\|f\|_{\widetilde{L}^{2}_{T_2} \widetilde{L}^{2}_{\xi,\nu} {(B^{3/2}_x)}}-\|f\|_{\widetilde{L}^{2}_{T_1} \widetilde{L}^{2}_{\xi,\nu} {(B^{3/2}_x)}}\right)\\
&\leq &\sum_{q\geq -1} 2^{\frac{3q}{2}} \left(\sup_{T_1\leq t\leq T_2}\|\Delta_qf (t)\|_{L^2_{x,\xi}}-\|\Delta_qf (T_1)\|_{L^2_{x,\xi}}\right)\\
&&+\sum_{q\geq -1} 2^{\frac{3q}{2}} \left(\int_{T_1}^{T_2}\|\Delta_qf(t)\|^2_{L^2_{\xi,\nu}L^2_x}dt\right)^{1/2}.
\end{eqnarray*}
Here, the second summation on the right tends to zero as $T_2\to T_1$ by \eqref{D.lim}, and in the completely same way to prove \eqref{D.lim}, one can see the first summation on the right also  tends to zero as $T_2\to T_1$, since $t\mapsto \|\Delta_qf(t)\|_{L^2_{x,\xi}}$ is continuous and $\|f\|_{\widetilde{L}^{\infty}_{T} \widetilde{L}^{2}_{\xi}(B^{3/2}_x)}$ is finite for some fixed time $T$ with $T_2<T<T^\ast$.
%

Finally, by using the iteration from \eqref{iterate.f}, the proof of the positivity is quite standard, for instance see \cite{Guo-BE-s}. We now finish the proof of Lemma \ref{local.existence}.
\end{proof}

\section{Proof of global Existence}
In this section, we prove the main result Theorem \ref{main.th.}  for the global existence of solutions to the Boltzmann equation \eqref{f.eq.} with initial data $f(0,x,\xi)=f_0(x,\xi)$. The approach is based on the local existence result Theorem \ref{local.existence} as well as the standard continuity argument.

\begin{proof}[Proof of Theorem \ref{main.th.}]
Recall \eqref{def.et} and \eqref{def.dt}.
Define
$$
{Y}_T(f)={\mathcal {E}}_T(f)+\mathcal {D}_T(f).
$$
Let us redefine the constant $C$ on the right of \eqref{basic.eng1.} to be $C_1\geq 1$, and choose $M_1>0$ such that
$$
C_1(\sqrt{M_1}+M_1)\leq \frac{1}{2}.
$$
Set
$$
M=\min\{M_{1}, M_{0}\}.
$$
Let initial data $f_0$ be chosen such that
\begin{equation}
\notag%
\|f_0\|_{\widetilde{L}^{2}_{\xi} B^{3/2}_x}\leq \frac{M}{4C_1}\leq \frac{M_0}{2}.
\end{equation}
Define
$$
\widetilde{T}=\sup\{T:{Y}_T(f)\leq M\}.
$$
By Theorem \ref{local.existence}, $\widetilde{T}>0$ holds true, because the solution $f$ exists locally in time and $T\mapsto Y_T(f)$ is continuous by the same proof as for $\widetilde{Y}_T(f)$. Moreover, Lemma \ref{basic.eng1.} gives that for $0\leq T\leq \widetilde{T}$,
\begin{equation}
\notag%
Y_T(f)\leq C_1\|f_0\|_{\widetilde{L}^{2}_{\xi} {(B^{3/2}_x)}}+C_1(\sqrt{M_1}+M_1) Y_T (f).
\end{equation}
That is, for $0\leq T\leq \widetilde{T}$,
\begin{equation}
\notag%
Y_T(f)\leq 2C_1\|f_0\|_{\widetilde{L}^{2}_{\xi} (B^{3/2}_x)}\leq \frac{M}{2}<M.
\end{equation}
This implies $\widetilde{T}=\infty$. The global existence and uniqueness are then proved. The proof of Theorem \ref{main.th.} is complete.
%
%
%
\end{proof}

\section{Appendix}\setcounter{equation}{0}
In this appendix, we will state some basic estimates related to the Botlzmann equation and Besov space. First we point out the coercivity property \eqref{coer} of the linearized collision operator $L$ implies
%

\begin{lemma}\label{qL+}
Let $(\cdot,\cdot)$ be the inner product on $L^2_{x,\xi}$. It holds that
\begin{equation*}
(\Delta_qLf, \Delta_qf)\geq\la_0\|\{\FI-\FP\}\Delta_qf\|_{L^2_{\xi,\nu}L^2_{x}}^{2},
\end{equation*}
for each $q\geq -1$.
Moreover, for $s\in \R$, it holds that
\begin{equation}\notag
\begin{split}
\sum\limits_{q\geq-1}2^{qs}\left(\int_0^T(\Delta_qLf, \Delta_qf)\,dt\right)^{1/2}
\geq\sqrt{\la_0}\|\{\FI-\FP\}f\|_{\widetilde{L}^{2}_T\widetilde{L}^{2}_{\xi,\nu}(B^{s}_x)},
\end{split}
\end{equation}
for any $T\geq 0$.
\end{lemma}

It is known that $K=K_2-K_1$ defined in \eqref{def.k1} and \eqref{K.def.} is a self-adjoint compact operator on $L^2_\xi$ (cf.~\cite{CIP-Book}) and it enjoys the following estimate.

\begin{lemma}\label{K.es.}
Let $(\cdot,\cdot)$ be the inner product on $L^2_{x,\xi}$.  It holds that
\begin{equation}\label{K.es.p1}
(\Delta_qKf, \Delta_qg)\leq C\|\Delta_qf\|_{L^2_{\xi}L^2_{x}}\|\Delta_qg\|_{L^2_{\xi}L^2_{x}},
\end{equation}
for each $q\geq -1$, where $C$ is a constant independent of $q$, $f$ and $g$.
\end{lemma}

\begin{proof}
$Kg$ can be written as
$$
Kg=\int_{\R^3}\mathcal {K}(\xi,\xi_\ast)g(\xi_\ast)\,d\xi_\ast,
$$
and $\mathcal {K}(\xi,\xi_\ast)$ is a bounded operator from $L_{\xi}^2$ to $L_{\xi}^2$. Then \eqref{K.es.p1} follows from
Cauchy-Schwarz inequality.
\end{proof}


{In addition,  for the  convenience of readers we list some basic facts which are  frequently used in the paper.}

\begin{lemma}
Let $1\leq p\leq\infty$, then
\begin{equation}\label{bdop}
\|\Delta_q\cdot\|_{L^p_x}\leq C\|\cdot\|_{L^p_x},\ \  \|S_{q}\cdot\|_{L^p_x}\leq C\|\cdot\|_{L^p_x},
\end{equation}
where $C$ is a constant independent of $p$ and $q$.
\end{lemma}

\begin{lemma}\label{B.IM.}
Let $1\leq \varrho, p, r\leq\infty$, if $s>0$, then
\begin{equation}\label{n-h-b}
\|\na_x\cdot\|_{\widetilde{L}^\varrho_T(\dot{B}_{p,r}^s)}\sim\|\cdot\|_{\widetilde{L}^\varrho_T(\dot{B}_{p,r}^{s+1})},\ \ \|\cdot\|_{\widetilde{L}^\varrho_T(\dot{B}_{p,r}^{s})}\lesssim\|\cdot\|_{\widetilde{L}^\varrho_T(B_{p,r}^{s})}.
\end{equation}
\end{lemma}

{We would like to mention that the first relation can be achieved by the classical Bernstein inequality (see, e.g., \cite{BCD}) and another follows from  the recent fact in \cite{XK}, which indicates the relation between homogeneous and
inhomogeneous Chemin-Lerner spaces.

{Finally, the Chemin-Lerner type spaces $\widetilde{L}_T^{\varrho_1}\widetilde{L}_\xi^{\varrho_2}(B^{s}_{p,r})$
may be linked with the classical spaces $L_T^{\varrho_1}L_\xi^{\varrho_2}(B^{s}_{p,r})$ in the following way.

\begin{lemma}\label{BCLsb}
Let $1\leq \varrho_1, \varrho_2, p, r\leq\infty$ and $s\in \R$.

\noindent$(1)$ If $r\geq \max\{\varrho_1, \varrho_2\}$, then
\begin{equation}\label{BCLs}
\|f\|_{\widetilde{L}_T^{\varrho_1}\widetilde{L}_\xi^{\varrho_2}(B^{s}_{p,r})}
\leq\|f\|_{L_T^{\varrho_1}L_\xi^{\varrho_2}(B^{s}_{p,r})}.
\end{equation}
$(2)$ If $r\leq \min\{\varrho_1, \varrho_2\}$, then
\begin{equation}\label{BCLb}
\|f\|_{\widetilde{L}_T^{\varrho_1}\widetilde{L}_\xi^{\varrho_2}(B^{s}_{p,r})}
\geq\|f\|_{L_T^{\varrho_1}L_\xi^{\varrho_2}(B^{s}_{p,r})}.
\end{equation}
\end{lemma}
\begin{proof}
We only prove \eqref{BCLb} in terms of $1\leq r, \varrho_2, \varrho_1<+\infty$, the other cases and \eqref{BCLs} can be proved similarly.

Since $\varrho_2/r\geq1$ and $\varrho_1/r\geq1$, by applying Generalized Minkowski's inequality twice, one can see that
\begin{equation*}\label{BCLb.p1}
\begin{split}
\|f\|_{\Red{L_T^{\varrho_1}}L_\xi^{\varrho_2}(B_{p,r}^s)}
=&\left(\int_{0}^T\left(
\int_{\R^3}\left(\sum_{q\geq-1}2^{qsr}\|\Delta_{q}f\|^r_{L_x^{p}}\right)^{\varrho_2/r}d\xi\right)
^{\varrho_1/\varrho_2}dt\right)^{1/\varrho_1}\\
=&\left(\int_{0}^T\left(
\int_{\R^3}\left(\sum_{q\geq-1}2^{qsr}\|\Delta_{q}f\|^r_{L_x^{p}}\right)^{\varrho_2/r}d\xi\right)
^{\frac{r}{\varrho_2}\cdot\frac{\varrho_1}{r}}dt\right)^{1/\varrho_1}\\
\leq&\left(\int_{0}^T\left(\sum_{q\geq-1}2^{qsr}\left(
\int_{\R^3}\|\Delta_{q}f\|^{\varrho_2}_{L_x^{p}}d\xi\right)
^{{r/\varrho_2}}\right)^{\frac{\varrho_1}{r}}dt\right)^{1/\varrho_1}\\
=&\left(\int_{0}^T\left(\sum_{q\geq-1}2^{qsr}\left(
\int_{\R^3}\|\Delta_{q}f\|^{\varrho_2}_{L_x^{p}}d\xi\right)
^{{r/\varrho_2}}\right)^{\frac{\varrho_1}{r}}dt\right)^{\frac{r}{\varrho_1}\cdot\frac{1}{r}}\\
\leq&\left(\sum_{q\geq-1}2^{qsr}\left(\int_{0}^T\left(
\int_{\R^3}\|\Delta_{q}f\|^{\varrho_2}_{L_x^{p}}d\xi\right)
^{\varrho_1/\varrho_2}dt\right)^{r/\varrho_1}\right)^{1/r}\\[2mm]
=&\|f\|_{{\widetilde{L}_T^{\varrho_1}}\widetilde{L}_\xi^{\varrho_2}(B^{s}_{p,r})}.
\end{split}
\end{equation*}
Thus Lemma \ref{BCLsb} holds true.
\end{proof}

\noindent {\bf Acknowledgements:} RJD was supported by the General Research Fund (Project No.~400912) from RGC of Hong Kong. SQL was
supported by grants from the National Natural Science Foundation of China under contracts 11101188 and 11271160. J. Xu was supported by
the NSFC (11001127), Program for New Century
Excellent Talents in University, and Special
Foundation of China Postdoctoral Science Foundation (2012T50466). The authors would delicate the paper to Professor Shuichi Kawashima on his 60th birthday.


\end{document}